\documentclass[12pt, reqno]{amsart}
\usepackage{mathrsfs}
\usepackage{amsfonts}
\usepackage[centertags]{amsmath}
\usepackage{amssymb}
\usepackage{amsthm}
\usepackage{graphicx}
\usepackage{caption}
\usepackage{hyperref}
\usepackage{enumerate}
\usepackage[textwidth=16cm, hmarginratio=1:1]{geometry}
\usepackage{appendix}
\usepackage[all]{xy}

\newcommand\scalemath[2]{\scalebox{#1}{\mbox{\ensuremath{\displaystyle #2}}}}

\newtheorem{theorem}{Theorem}[section]
\newtheorem{corollary}[theorem]{Corollary}
\newtheorem{lemma}[theorem]{Lemma}
\newtheorem{proposition}[theorem]{Proposition}
\newtheorem{definition}[theorem]{Definition}

\newtheorem{example}[theorem]{Example}

\numberwithin{equation}{section}

\DeclareMathOperator*{\rep}{Rep}
\DeclareMathOperator*{\diag}{diag}
\DeclareMathOperator*{\res}{Res}

\begin{document}

\title[Three-term recurrence relations of minimal affinizations of type $G_2$]{Three-term recurrence relations of minimal affinizations of type $G_2$}
\author{Li Qiao and Jian-Rong Li}
\address{Li Qiao: Department of Mathematics, Lanzhou University, Lanzhou 730000, P. R. China.}
\email{qiaol12@lzu.edu.cn}
\address{Jian-Rong Li, Dept. of Mathematics, The Weizmann Institute of Science, Rehovot 7610001, Israel; Einstein Institute of Mathematics, The Hebrew University of Jerusalem, Jerusalem 9190401, Israel; and School of Mathematics and Statistics, Lanzhou University, Lanzhou 730000, P. R. China.}
\email{lijr07@gmail.com}
\date{}

\maketitle

\begin{abstract}
Minimal affinizations form a class of modules of quantum affine algebras introduced by Chari. We introduce a system of equations satisfied by the $q$-characters of minimal affinizations of type $G_2$ which we call the M-system of type $G_2$. The M-system of type $G_2$ contains all minimal affinizations of type $G_2$ and only contains minimal affinizations. The equations in the M-system of type $G_2$ are three-term recurrence relations. The M-system of type $G_2$ is much simpler than the extended T-system of type $G_2$ obtained by Mukhin and the second author. We also interpret the three-term recurrence relations in the M-system of type $G_2$ as exchange relations in a cluster algebra constructed by Hernandez and Leclerc.

\hspace{0.15cm}


\hspace{0.15cm}

\end{abstract}

\section{Introduction}
Let $\mathfrak{g}$ be a simple Lie algebra and $U_q \widehat{\mathfrak{g}}$ the corresponding quantum affine algebra. Minimal affinizations are simple modules of $U_q \widehat{\mathfrak{g}}$ which were introduced by Chari in \cite{C95}. The family of minimal affinizations contains the celebrated Kirillov-Reshetikhin modules. Minimal affinizations are studied intensively in recent years, see for example, \cite{CMY13}, \cite{CG11}, \cite{H07}, \cite{LM13}, \cite{LN15}, \cite{M10}, \cite{MP11}, \cite{MY12a}, \cite{MY12b}, \cite{MY14}, \cite{Nao13}, \cite{ZDLL15}.

The aim of this paper is to study three-term recurrence relations satisfied by the $q$-characters of minimal affinizations of type $G_2$. The set of minimal affinizations of type $G_2$ can be divided into two sets $X_1$, $X_2$ according to their highest $l$-weights. The minimal affinizations in $X_1$ have highest $l$-weight monomials of the form (see Section \ref{definition of minimal affinizations})
\begin{align*}
T_{k, l}^{(s)} =\left( \prod_{i=0}^{k-1} 1_{s+6i} \right) \left( \prod_{j=0}^{l-1} 2_{s+6k+2j+1} \right)
\end{align*}
and the minimal affinizations in $X_2$ have highest $l$-weight monomials of the form
\begin{align*}
\widetilde{T}_{k, l}^{(s)} =\left(\prod_{i=0}^{l-1} 2_{-s-6k-2i-1}\right)\left(\prod_{j=0}^{k-1} 1_{-s-6j}\right).
\end{align*}

We introduce a system of equations which we call the M-system of type $G_2$ and prove that the equations in the M-system of type $G_2$ are satisfied by the $q$-characters of minimal affinizations.

The equations in the first part of the M-system of type $G_2$ are three-term recurrence relations:
\begin{eqnarray}
[\mathcal T_{k, l}^{(s)}][\mathcal T_{k, 0}^{(s+6)}]&=&[\mathcal T_{k+1, 0}^{(s)}][\mathcal T_{k-1, l}^{(s+6)}]+[\mathcal T_{0, 3k+l}^{(s)}] \quad (k \in \mathbb{Z}_{\geq 1}, l \in \{1, 2, 3\}), \\
{[\mathcal T_{k, l+3}^{(s)}]}[\mathcal T_{k, l}^{(s+6)}]&=&[\mathcal T_{k+1, l}^{(s)}][\mathcal T_{k-1, l+3}^{(s+6)}]+[\mathcal T_{0, l}^{(s+6k+6)}][\mathcal T_{0, 3k+l+3}^{(s)}] \quad (k, l \in \mathbb{Z}_{\geq 1}),
\end{eqnarray}
see Theorem \ref{first part of the M system of type G2}. They are satisfied by the $q$-characters of the minimal affinizations in $X_1$. Here we use $\mathcal{T}$ to denote a module with highest $l$-weight $T$.

The equations in the second part of the M-system of type $G_2$ are three-term recurrence relations:
\begin{eqnarray*}
[\widetilde{\mathcal T}_{k, l}^{(s)}][\widetilde{\mathcal T}_{k, 0}^{(s+6)}]&=&[\widetilde{\mathcal T}_{k+1, 0}^{(s)}][\widetilde{\mathcal T}_{k-1, l}^{(s+6)}]+[\widetilde{\mathcal T}_{0, 3k+l}^{(s)}] \quad (k \in \mathbb{Z}_{\geq 1}, l \in \{1, 2, 3\}), \\
{[\widetilde{\mathcal T}_{k, l+3}^{(s)}]}[\widetilde{\mathcal T}_{k, l}^{(s+6)}]&=&[\widetilde{\mathcal T}_{k+1, l}^{(s)}][\widetilde{\mathcal T}_{k-1, l+3}^{(s+6)}]+[\widetilde{\mathcal T}_{0, l}^{(s+6k+6)}][\widetilde{\mathcal T}_{0, 3k+l+3}^{(s)}] \quad (k, l\in \mathbb{Z}_{\geq 1}),
\end{eqnarray*}
see Theorem \ref{second part of M system}. They are satisfied by the $q$-characters of the minimal affinizations in $X_2$.

The extended T-system of type $G_2$ obtained by Mukhin and the second author in \cite{LM13} contains all minimal affinizations of type $G_2$ and some other modules which are not minimal affinizations. The M-system of type $G_2$ also contains all minimal affinizations of type $G_2$. But unlike the extended T-system of type $G_2$, the M-system of type $G_2$ contains only minimal affinizations of type $G_2$. The M-system of type $G_2$ is much simpler than the extended T-system of type $G_2$.

The equations the M-system of type $G_2$ can be interpreted as exchange relations in a certain cluster algebra $\mathscr{A}$ constructed by Hernandez and Leclerc in \cite{HL16}, see Section \ref{interpret M systems as exchange relations}. In the paper \cite{HL16}, the equations in the usual T-systems are interpreted as exchange relations in some cluster algebras. The T-system of type $G_2$ and the M-system of type $G_2$ are special cases of exchange relations in the cluster algebra $\mathscr{A}$.

We also used the M-system of type $G_2$ to compute the decomposition of a minimal affinization of type $G_2$ as a $U_q\mathfrak{g}$-module into simple $U_q\mathfrak{g}$-modules. This helps us to obtain the general decomposition formula in \cite{LN15}.

We show that the modules associated to the summands on the right hand side of each equation in the M-system are simple.

The paper is organized as follows. In Section 2, we give some background information about finite-dimensional representations of quantum affine algebras and cluster algebras. In Section \ref{sec: first part of the M system of type G2}, we describe the M-system of type $G_2$. In Section \ref{interpret M systems as exchange relations}, we interpret the equations in the M-system of type $G_2$ as exchange relations. In Section \ref{proof main1} and \ref{proof simplicity of the modules on the right hand side} we prove Theorem \ref{first part of the M system of type G2}. In Section \ref{sec: second part of M system}, we prove Theorem \ref{second part of M system}.

\section{Background}

\subsection{The quantum affine algebra of type $G_2$}

In this paper, we take $\mathfrak{g}$ to be the complex simple Lie algebra of type $G_2$ and $\mathfrak{h}$ a Cartan subalgebra of $\mathfrak{g}$. Let $I=\{1, 2\}$.
We choose simple roots $\alpha_1, \alpha_2$ and scalar product $(\cdot, \cdot)$ such that
\begin{align*}
( \alpha_1, \alpha_1 ) = 6, \ ( \alpha_1, \alpha_2 )=-3, \ ( \alpha_2, \alpha_2 )=2.
\end{align*}
Therefore $\alpha_1$ is the long simple root and $\alpha_2$ is the short simple root. Let $\{\alpha_1^{\vee}, \alpha_2^{\vee}\}$ and $\{\omega_1, \omega_2\}$ be the sets of simple coroots and fundamental weights respectively. Let $C=(C_{ij})_{i, j\in I}$ denote the Cartan matrix, where $C_{ij} =\frac{2 ( \alpha_i, \alpha_j ) }{( \alpha_i, \alpha_i )}$. Let $r_1=3, r_2=1$, $D=\diag (r_1, r_2)$ and $B=DC$.
Then
\begin{align*}
C=
\left(
\begin{array}{cc}
2 & -1 \\
-3 & 2
\end{array}
\right), \quad
B=
\left(
\begin{array}{cc}
6 & -3 \\
-3 & 2
\end{array}
\right).
\end{align*}

Let $Q$ (resp. $Q^+$) and $P$ (resp. $P^+$) denote the $\mathbb{Z}$-span (resp. $\mathbb{Z}_{\geq 0}$-span) of the simple roots and fundamental weights respectively. Let $\leq$ be the partial order on $P$ in which $\lambda \leq \lambda'$ if and only if $\lambda' - \lambda \in Q^+$.

Let $\widehat{\mathfrak{g}}$ denote the untwisted affine algebra corresponding to $\mathfrak{g}$. Fix a $q\in \mathbb{C}^{\times}$, not a root of unity. Let $q_i=q^{r_i}, i=1, 2$. Let $\mathcal{P}$ the free abelian multiplicative group of monomials in infinitely many formal variables $(Y_{i, a})_{i\in I, a \in \mathbb{C}^{\times}}$.

The quantum affine algebra $U_q\widehat{\mathfrak{g}}$ in Drinfeld's new realization, see \cite{Dri88}, is generated by $x_{i, n}^{\pm}$ ($i\in I, n\in \mathbb{Z}$), $k_i^{\pm 1}$ $(i\in I)$, $h_{i, n}$ ($i\in I, n\in \mathbb{Z}\backslash \{0\}$) and central elements $c^{\pm 1/2}$, subject to certain relations.

The quantum affine algebra $U_q\widehat{\mathfrak{g}}$ contains two standard quantum affine algebras of type $A_1$. The first one is $U_{q_1} \widehat{\mathfrak{sl}}_2$ generated by $x_{1, n}^{\pm}$ ($n\in \mathbb{Z}$), $k_1^{\pm 1}$, $h_{1, n}$ ($n\in \mathbb{Z}\backslash \{0\}$) and central elements $c^{\pm 1/2}$. The second one is $U_{q_2} \widehat{\mathfrak{sl}}_2 $ generated by $x_{2, n}^{\pm}$ ($n\in \mathbb{Z}$), $k_2^{\pm 1}$, $h_{2, n}$ ($n\in \mathbb{Z}\backslash \{0\}$) and central elements $c^{\pm 1/2}$.

The subalgebra of $U_q\widehat{\mathfrak{g}}$ generated by $(k_i^{\pm})_{i\in I}, (x_{i, 0}^{\pm})_{i\in I}$ is a Hopf subalgebra of $U_q\widehat{\mathfrak{g}}$ and is isomorphic as a Hopf algebra to $U_q\mathfrak{g}$. Therefore $U_q\widehat{\mathfrak{g}}$-modules restrict to $U_q\mathfrak{g}$-modules.

\subsection{Finite-dimensional representations of $U_q \widehat{\mathfrak{g}}$ and $q$-characters} In this section, we recall the standard facts about finite-dimensional $U_q\widehat{\mathfrak{g}}$-modules and $q$-characters of these representations, see \cite{CP94}, \cite{CP95a}, \cite{FR98}, \cite{MY12a}.

A representation $V$ of $U_q\widehat{\mathfrak{g}}$ is of type $1$ if $c^{\pm 1/2}$ acts as the identity on $V$ and
\begin{align} \label{decomposition}
V=\bigoplus_{\lambda \in P} V_{\lambda}, \ V_{\lambda} = \{v\in V  :   k_i v=q^{( \alpha_i, \lambda )} v \}.
\end{align}
In the following, all representations will be assumed to be finite-dimensional and of type $1$ without further comment. The decomposition (\ref{decomposition}) of a finite-dimensional representation $V$ into its $U_q\mathfrak{g}$-weight spaces can be refined by decomposing it into the Jordan subspaces of the mutually commuting operators $\phi_{i, \pm r}^{\pm}$, see \cite{FR98}:
\begin{align}
V=\bigoplus_{\gamma} V_{\gamma}, \ \gamma=(\gamma_{i, \pm r}^{\pm})_{i\in I, r\in \mathbb{Z}_{\geq 0}}, \ \gamma_{i, \pm r}^{\pm} \in \mathbb{C},
\end{align}
where
\begin{align*}
V_{\gamma} = \{v\in V  :   \exists k\in \mathbb{N}, \forall i \in I, m \geq 0, (\phi_{i, \pm m}^{\pm} - \gamma_{i, \pm m}^{\pm})^{k} v =0 \}.
\end{align*}
Here $\phi_{i, n}^{\pm}$'s are determined by the formula
\begin{align}
\phi_i^{\pm}(u) := \sum_{n=0}^{\infty} \phi_{i, \pm n}^{\pm} u^{\pm n} = k_i^{\pm 1} \exp\left( \pm (q-q^{-1}) \sum_{m =1}^{\infty} h_{i, \pm m}u^{\pm m}\right).
\end{align}
If $\dim(V_{\gamma})>0$, then $\gamma$ is called an \textit{$l$-weight} of $V$. Let $\gamma$ be the $l$-weight of a finite dimensional $U_q\widehat{\mathfrak{g}}$-module. In \cite{FR98}, it is shown $\gamma$ satisfies
\begin{align}
\gamma_i^{\pm}(u) := \sum_{r=0}^{\infty} \gamma_{i, \pm r}^{\pm} u^{\pm r} = q_i^{\deg Q_i - \deg R_i} \frac{Q_i(uq_i^{-1})R_i(uq_i)}{Q_{i}(uq_i) R_{i}(uq_i^{-1})}, \label{gamma}
\end{align}
where the right hand side is to be treated as a formal series in positive (resp. negative) integer powers of $u$, and $Q_i, R_i$ are polynomials of the form
\begin{align}
Q_i(u) = \prod_{a \in \mathbb{C}^{\times}} (1-ua)^{w_{i, a}}, \ R_{i}(u)=\prod_{a \in \mathbb{C}^{\times}}(1-ua)^{x_{i, a}}, \label{QR}
\end{align}
for some $w_{i, a}, x_{i, a} \in \mathbb{Z}_{\geq 0}, i\in I, a \in \mathbb{C}^{\times}$. Let $\mathcal{P}$ denote the free abelian multiplicative group of monomials in infinitely many formal variables $(Y_{i, a})_{i\in I, a \in \mathbb{C}^{\times}}$. There is a bijection $\gamma$ from $\mathcal{P}$ to the set of $l$-weights of finite-dimensional modules such that for the monomial $m =\prod_{i\in I, a \in \mathbb{C}^{\times}} Y_{i, a}^{w_{i, a}-x_{i, a}}$, the $l$-weight $\gamma(m)$ is given by (\ref{gamma}), (\ref{QR}).

Let $\mathbb{Z}\mathcal{P} = \mathbb{Z}[Y_{i, a}^{\pm 1}]_{i\in I, a \in \mathbb{C}^{\times}}$ be the group ring of $\mathcal{P}$. For $\chi \in \mathbb{Z}\mathcal{P}$, we write $m\in \mathcal{P}$ if the coefficient of $m$ in $\chi$ is non-zero.

The $q$-character of a $U_q\widehat{\mathfrak{g}}$-module $V$ is defined by
\begin{align*}
\chi_q(V) = \sum_{m\in \mathcal{P}} \dim(V_{m}) m \in \mathbb{Z}\mathcal{P},
\end{align*}
where $V_m = V_{\gamma(m)}$, see \cite{FR98}.

Let $\rep(U_q\widehat{\mathfrak{g}})$ be the Grothendieck ring of finite-dimensional $U_q\widehat{\mathfrak{g}}$-modules and $[V]\in \rep(U_q\widehat{\mathfrak{g}})$ the class of a finite-dimensional $U_q\widehat{\mathfrak{g}}$-module $V$. The $q$-character map defines an injective ring homomorphism, see \cite{FR98},
\begin{align*}
\chi_q: \rep(U_q\widehat{\mathfrak{g}}) \to \mathbb{Z}\mathcal{P}.
\end{align*}

For any finite-dimensional $U_q\widehat{\mathfrak{g}}$-module $V$, we use $m \in \chi_q(V)$ to denote that $m$ is a monomial in $\chi_q(V)$. For each $j\in I$, a monomial $m =\prod_{i\in I, a \in \mathbb{C}^{\times}} Y_{i, a}^{u_{i, a}}$, where $u_{i, a}$ are some integers, is said to be \textit{$j$-dominant} (resp. \textit{$j$-anti-dominant}) if and only if $u_{j, a} \geq 0$ (resp. $u_{j, a} \leq 0$) for all $a \in \mathbb{C}^{\times}$. A monomial is called \textit{dominant} (resp. \textit{anti-dominant}) if and only if it is $j$-dominant (resp. $j$-anti-dominant) for all $j\in I$. Let $\mathcal{P}^+ \subset \mathcal{P}$ denote the set of all dominant monomials.

Let $V$ be a $U_q\widehat{\mathfrak{g}}$-module and $m\in \chi_q(V)$ a monomial. A non-zero vector $v\in V_m$ is called a \textit{highest $l$-weight vector} with \textit{highest $l$-weight} $\gamma(m)$ if
\begin{align*}
x_{i, r}^{+} \cdot v=0, \ \phi_{i, \pm t}^{\pm} \cdot v=\gamma(m)_{i, \pm t}^{\pm} v, \ \forall i\in I, r\in \mathbb{Z}, t\in \mathbb{Z}_{\geq 0}.
\end{align*}
The module $V$ is called a \textit{highest $l$-weight module} if $V=U_q\widehat{\mathfrak{g}}\cdot v$ for some highest $l$-weight vector $v\in V$.

In \cite{CP94}, \cite{CP95a}, it is shown that there is a one to one correspondence between dominant $l$-weights and finite-dimensional simple $U_q\widehat{\mathfrak{g}}$-modules. Therefore for every $m_+ \in \mathcal{P}^+$, there is a unique finite-dimensional simple $U_q\widehat{\mathfrak{g}}$-module $L(m_+)$. We use $\chi_q(m_+)$ to denote $\chi_q(L(m_+))$.

Let $p_1, p_2$ be two polynomials in $\mathbb{Z}[Y_{i,a}^{\pm 1}]_{i \in I, a \in \mathbb{C}^{\times}}$. If $m$ is a monomial in the polynomial, then we write $m \in p_1$. If $m \in p_1$ and $m \in p_2$, then we write $m \in p_1 \cap p_2$. If all monomials in $p_1$ are in $p_2$, then we write $p_1 \subseteq p_2$.

The following lemma is well-known.
\begin{lemma}
\label{contains in a larger set}
Let $m_1, m_2$ be two dominant monomials. Then $L(m_1m_2)$ is a sub-quotient of $L(m_1) \otimes L(m_2)$. In particular, $\chi_q(m_1m_2) \subseteq \chi_q(m_1 ) \chi_q( m_2)$.   $\Box$
\end{lemma}

For $b\in \mathbb{C}^{\times}$, define the shift of spectral parameter map $\tau_b: \mathbb{Z}\mathcal{P} \to \mathbb{Z}\mathcal{P}$ to be a homomorphism of rings sending $Y_{i, a}^{\pm 1}$ to $Y_{i, ab}^{\pm 1}$. Let $m_1, m_2 \in \mathcal{P}^+$. If $\tau_b(m_1) = m_2$, then
\begin{align} \label{shift}
\tau_b \chi_q(m_1) = \chi_q(m_2).
\end{align}

A finite-dimensional $U_q\widehat{\mathfrak{g}}$-module $V$ is said to be \textit{special} if and only if $\chi_q(V)$ contains exactly one dominant monomial. It is called \textit{anti-special} if and only if $\chi_q(V)$ contains exactly one anti-dominant monomial. It is said to be \textit{prime} if and only if it is not isomorphic to a tensor product of two non-trivial $U_q\widehat{\mathfrak{g}}$-modules, see \cite{CP97}. Clearly, if a module is special or anti-special, then it is simple.

Define $A_{i, a} \in \mathcal{P}, i\in I, a \in \mathbb{C}^{\times}$, by
\begin{align*}
A_{1, a} = Y_{1, aq^{3}}Y_{1, aq^{-3}} Y_{2, aq^{-2}}^{-1} Y_{2, a}^{-1} Y_{2, aq^{2}}^{-1}, \qquad
A_{2, a} = Y_{2, aq}Y_{2, aq^{-1}} Y_{1, a}^{-1}.
\end{align*}
When $a \in \mathbb{C}^{\times}$ is fixed, we write $A_{i, s}=A_{i, aq_i^{s}}$.

Let $\mathcal{Q}$ be the subgroup of $\mathcal{P}$ generated by $A_{i, a}, i\in I, a \in \mathbb{C}^{\times}$. Let $\mathcal{Q}^{\pm}$ be the monoids generated by $A_{i, a}^{\pm 1}, i\in I, a \in \mathbb{C}^{\times}$. There is a partial order $\leq$ on $\mathcal{P}$ in which
\begin{align}
m\leq m' \text{ if and only if } m'm^{-1} \in \mathcal{Q}^{+}. \label{partial order of monomials}
\end{align}
For all $m_+ \in \mathcal{P}^+$, $\chi_q(m_+) \subset m_+\mathcal{Q}^{-}$, see \cite{FM01}.

\begin{definition} [{\cite{FM01}}]
Let $m$ be a monomial. Suppose that for all $a \in \mathbb{C}^{\times}$ and $i\in I$, we have the property: if the power of $Y_{i, a}$ in $m$ is non-zero and the power of $Y_{j, aq^k}$ in $m$ is zero for all $j \in I, k\in \mathbb{Z}_{>0}$, then the power of $Y_{i, a}$ in $m$ is negative. Then the monomial $m$ is called \textit{right negative}.
\end{definition}

\begin{lemma} [{\cite{FM01}, \cite{H07}}] \label{properties of right negative monomials}
For $i\in I, a \in \mathbb{C}^{\times}$, $A_{i,a}^{-1}$ is right-negative. A product of right-negative monomials is right-negative. If $m$ is right-negative and $m'\leq m$, then $m'$ is right-negative.
\end{lemma}

\begin{lemma} [{Lemma 4.4, \cite{H06}}] \label{non highest monomials in q characters of KR modules are right negative}
All monomials in the $q$-character of a Kirillov-Reshetikhin module is right-negative except the highest $l$-weight monomial.
\end{lemma}

We need the following result from \cite{FM01}, \cite{HL10}.
\begin{proposition}[{Proposition 5.3, \cite{HL10}}] \label{dominant monomials determine q-characters}
Let $V, W$ be two $U_q\widehat{\mathfrak{g}}$-modules. If $\chi_q(V)$ and $\chi_q(W)$ have the same dominant monomials with the same multiplicities, then $\chi_q(V) = \chi_q(W)$.
\end{proposition}

\subsection{Minimal affinizations of $U_q \mathfrak{g}$-modules} \label{definition of minimal affinizations}
Let $\lambda = k\omega_1 + l \omega_2$. A simple $U_q \widehat{\mathfrak{g}} $-module $L(m_+)$ is a \textit{minimal affinization} of $V(\lambda)$ if and only if $m_+$ is one of the following monomials
\begin{align*}
\left( \prod_{i=0}^{k-1} Y_{1, aq^{6i}} \right) \left( \prod_{i=0}^{l-1} Y_{2, aq^{6k+2i+1}} \right), \qquad
\left(\prod_{i=0}^{l-1}Y_{2, aq^{-6k-2i-1}}\right)\left(\prod_{j=0}^{k-1}Y_{1, aq^{-6j}}\right),
\end{align*}
for some $a \in \mathbb{C}^{\times}$, see \cite{CP95b}.

From now on, we fix an $a \in \mathbb{C}^{\times}$ and denote $i_s = Y_{i, aq^s}$, $i \in I$, $s \in \mathbb{Z}$. Without loss of generality, we may assume that a simple $U_q \widehat{\mathfrak{g}} $-module $L(m_+)$ is a minimal affinization of $V(\lambda)$ if and only if $m_+$ is one of the following monomials
\begin{align*}
T_{k, l}^{(s)} =\left( \prod_{i=0}^{k-1} 1_{s+6i} \right) \left( \prod_{j=0}^{l-1} 2_{s+6k+2j+1} \right), \quad
\widetilde{T}_{k, l}^{(s)} =\left(\prod_{i=0}^{l-1} 2_{-s-6k-2i-1}\right)\left(\prod_{j=0}^{k-1} 1_{-s-6j}\right).
\end{align*}

\subsection{$q$-characters of $U_q \widehat{\mathfrak{sl}}_2$-modules and the Frenkel-Mukhin algorithm}
We recall the results of the $q$-characters of $U_q \widehat{\mathfrak{sl}}_2$-modules and Frenkel-Mukhin algorithm, see \cite{CP91}, \cite{FR98}, \cite{FM01}, \cite{H05}, \cite{H08}.

Let $W_{k}^{(a)}$ be the simple $U_q \widehat{\mathfrak{sl}}_2$-module with
highest weight monomial
\begin{align*}
X_{k}^{(a)} =\prod_{i=0}^{k-1} Y_{aq^{k-2i-1}},
\end{align*}
where $Y_a=Y_{1, a}$. Then the $q$-character of $W_{k}^{(a)}$ is given by
\begin{align} \label{q-characters of Uqsl_2 KR module}
\chi_q(W_{k}^{(a)})=X_{k}^{(a)} \sum_{i=0}^{k} \prod_{j=0}^{i-1} A_{aq^{k-2j}}^{-1},
\end{align}
where $A_a=Y_{aq^{-1}}Y_{aq}$.

For $a \in \mathbb{C}^{\times}, k\in \mathbb{Z}_{\geq 1}$, the set $\Sigma_{k}^{(a)} =\{aq^{k-2i-1}\}_{i=0, \ldots, k-1}$ is called a \textit{$q$-string}. Two $q$-strings $\Sigma_{k}^{(a)}$ and $\Sigma_{k'}^{(a')}$ are said to be in \textit{general position} if the union $\Sigma_{k}^{(a)} \cup \Sigma_{k'}^{(a')}$ is not a $q$-string or $\Sigma_{k}^{(a)} \subset \Sigma_{k'}^{(a')}$ or $\Sigma_{k'}^{(a')} \subset \Sigma_{k}^{(a)}$.

Denote by $L(m_+)$ the simple $U_q \widehat{\mathfrak{sl}}_2$-module with
highest weight monomial $m_+$. Let $m_{+} \neq 1$ and $\in \mathbb{Z}[Y_a]_{a \in \mathbb{C}^{\times}}$ be a dominant monomial. Then $m_+$ can be uniquely (up to permutation) written in the form
\begin{align*}
m_+=\prod_{i=1}^{s} \left( \prod_{b\in \Sigma_{k_i}^{(a_i)}} Y_{b} \right),
\end{align*}
where $s$ is an integer, $\Sigma_{k_i}^{(a_i)}, i=1, \ldots, s$, are $q$-strings which are pairwise in general position and
\begin{align} \label{q-characters of Uqsl_2 module}
L(m_+)=\bigotimes_{i=1}^s W_{k_i}^{(a_i)}, \qquad \chi_q(L(m_+))=\prod_{i=1}^s \chi_q(W_{k_i}^{(a_i)}).
\end{align}

Let $i \in I$. We also call $n=Y_{i, a}Y_{i, aq_i^{2}} \cdots Y_{i, aq_i^{2k-2}}$ a $q_i$-string in a monomial $m$ if $n$ is a factor of $m$. We say that two $q_i$-strings $n_1$ and $n_2$ are in general position if $n_1 n_2$ is not a $q_i$-string or $n_1$ is a factor of $n_2$ or $n_2$ is a factor of $n_1$.

The Frenkel-Mukhin algorithm is very powerful to compute $q$-characters of simple $U_q \mathfrak{g}$-modules, \cite{FM01}. Let $m_+$ be a dominant monomial. Roughly speaking, when the Frenkel-Mukhin algorithm computes $\chi_q(m_+)$, the algorithm starts with $m_+$ and gradually expand it in all possible $U_{q_i}\widehat{\mathfrak{sl}}_2$-directions ($i \in I$).

Although in some cases the algorithm may fail, it works for a large family of modules. In particular, if a module $L(m_+)$ is special, then we can use Frenkel-Mukhin algorithm to compute its $q$-character, see \cite{FM01}.

\begin{theorem}[{Theorem 3.8, \cite{H07}, Proposition 7.1, Theorem 7.2, \cite{LM13}}] \label{minimal affinizations in X1 are special}
The minimal affinizations in the first (resp. second) part of the M-system of type $G_2$ are special (resp. anti-special). Therefore we can use the Frenkel-Mukhin algorithm to compute the $q$-characters of the minimal affinizations in the first part of the M-system of type $G_2$.
\end{theorem}

We will need the following result from Section 5 of \cite{HL10}. Let $m$ be an $i$-dominant monomial and $\varphi_i(m)$ a polynomial defined as follows. Let $\overline{m}$ be the monomial obtained from $m$ by replacing $Y_{j,a}$ by $Y_a$ if $j = i$ and by $1$ if $j \neq i$. Then the $q$-character $\chi_q(L(\overline{m}))$ of the $U_q \widehat{\mathfrak{sl}}_2$-module $L(\overline{m})$ is given by (\ref{q-characters of Uqsl_2 KR module}), (\ref{q-characters of Uqsl_2 module}). Write $\chi_q(L(\overline{m})) = \overline{m}(1 + \sum_{p} \overline{M}_p)$, where the $\overline{M}_p$ are monomials in the variables $A^{-1}_{a}$ $(a \in \mathbb{C}^{\times})$. Let $\varphi_i(m) := m(1 + \sum_p M_p)$ where each $M_p$ is obtained from the corresponding $\overline{M}_p$ by replacing each variable $A^{-1}_{a}$ by $A^{-1}_{i,a}$.

\begin{theorem}[{Section 5.3, \cite{HL10}}] \label{Uqsl_2 arguments}
Let $m$ be a dominant monomial and let $mM$ be a monomial of $\chi_q(L(m))$, where $M$ is a monomial in $A_{j,a}^{-1}$ ($j \in I$). If $M$ contains no $A_{i,a}^{-1}$, then
$mM$ is $i$-dominant and $\varphi_i(mM)$ is contained in $\chi_q(L(m))$. In particular, $\varphi_i(m)$ is contained in $\chi_q(L(m))$.
\end{theorem}

Let $i \in I$ and $\beta_i: \mathcal{P} \to \mathcal{P}$ be a map such that $\beta_i(m)$ is obtained from $m \in \mathcal{P}$ by replacing all $Y_{j,a}$ by $1$, $j \neq i$. For example, $\beta_1(1_0 1_6 1_{12}^{-1} 2_1 2_3) = 1_0 1_6 1_{12}^{-1}$.

By the Frenkel-Mukhin algorithm \cite{FM01} and the formulas (\ref{q-characters of Uqsl_2 KR module}), (\ref{q-characters of Uqsl_2 module}), we have the following result which is used frequently in our proof.
\begin{lemma} \label{need to expand from right most factors}
Let $m_+$ be a dominant monomial. Then every monomial in $\chi_q(m_+)$ is a monomial in some $\varphi_i(m)$, where $i \in I$ and $m$ is an $i$-dominant monomial in $\chi_q(m_+)$. The $l$-weights of the monomials in $\varphi_i(m)$ are less or equal to the $l$-weight of $m$.

Suppose that $\beta_i(m)=i_s i_{s+2r_i} \cdots i_{s+2kr_i-2r_i}$ ($k \in \mathbb{Z}_{\geq 1}$) is a $q_i$-string and $m' \in \varphi_i(m)$. If $i_{s+2kr_i-2r_i}$ is a factor of $m'$, then $\beta_i(m') = \beta_i(m)$ and hence $m A_{i, s+2jr_i-r_i}^{-1}$ ($j \in \{1,\ldots,k\}$) is not a monomial in $\chi_q(m_+)$.
\end{lemma}
For example, $1_0 2_7 2_9 A_{2, aq^8}^{-1} = 1_0 1_{9}$ is not in $\chi_q(1_0 2_7 2_9)$.

\subsection{Cluster algebras}
Cluster algebras are invented by Fomin and Zelevinsky in \cite{FZ02}. Let $\mathbb{Q}$ be the field of rational numbers and $\mathcal{F} =\mathbb{Q}(x_{1}, x_{2}, \cdots, x_{n})$ the field of rational functions. A seed in $\mathcal{F}$ is a pair $\Sigma=({\bf y}, Q)$, where ${\bf y} = (y_{1}, y_{2}, \cdots, y_{n})$ is a free generating set of $\mathcal{F}$, and $Q$ is a quiver with vertices labeled by $\{1, 2, \cdots, n\}$. Assume that $Q$ has neither loops nor $2$-cycles. For $k=1, 2, \cdots, n$, one defines a mutation $\mu_k$ by $\mu_k({\bf y}, Q) = ({\bf y}', Q')$. Here ${\bf y}' = (y_1', \ldots, y_n')$, $y_{i}'=y_{i}$, for $i\neq k$, and
\begin{equation}
y_{k}'=\frac{\prod_{i\rightarrow k} y_{i}+\prod_{k\rightarrow j} y_{j}}{y_{k}}, \label{exchange relation}
\end{equation}
where the first (resp. second) product in the right hand side is over all arrows of $Q$ with target (resp. source) $k$, and $Q'$ is obtained from $Q$ by
\begin{enumerate}
\item[(i)] adding a new arrow $i\rightarrow j$ for every existing pair of arrow $i\rightarrow k$ and $k\rightarrow j$;

\item[(ii)] reversing the orientation of every arrow with target or source equal to $k$;

\item[(iii)] erasing every pair of opposite arrows possible created by (i).
\end{enumerate}
The mutation class $\mathcal{C}(\Sigma)$ is the set of all seeds obtained from $\Sigma$ by a finite sequence of mutation $\mu_{k}$. If $\Sigma'=((y_{1}', y_{2}', \cdots, y_{n}'), Q')$ is a seed in $\mathcal{C}(\Sigma)$, then the subset $\{y_{1}', y_{2}', \cdots, y_{n}'\}$ is called a $cluster$, and its elements are called \textit{cluster variables}. The \textit{cluster algebra} $\mathscr{A}_{\Sigma}$ as the subring of $\mathcal{F}$ generated by all cluster variables. \textit{Cluster monomials} are monomials in the cluster variables supported on a single cluster.

In this paper, the initial seed in the cluster algebra we use is of the form $\Sigma=({\bf y}, Q)$, where ${\bf y}$ is an infinite set and $Q$ is an infinite quiver.

\begin{definition}[{Definition 3.1, \cite{GG14}}] \label{definition of cluster algebras of infinite rank}
Let $Q$ be a quiver without loops or $2$-cycles and with a countably infinite number of vertices labelled by all integers $i \in \mathbb{Z}$. Furthermore, for each vertex $i$ of $Q$ let the
number of arrows incident with $i$ be finite. Let ${\bf y} = \{y_i \mid i \in \mathbb{Z}\}$. An infinite initial seed is the pair $({\bf y}, Q)$. By finite sequences of mutation
at vertices of $Q$ and simultaneous mutation of the set ${\bf y}$ using the exchange relation (\ref{exchange relation}), one obtains a family of infinite seeds. The sets of variables in these seeds are called
the infinite clusters and their elements are called the cluster variables. The cluster algebra of
infinite rank of type $Q$ is the subalgebra of $\mathbb{Q}({\bf y})$ generated by the cluster variables.
\end{definition}

\section{The M-system of type $G_2$} \label{sec: first part of the M system of type G2}
In this section, we describe the M-system of type $G_2$.

\subsection{The M-system of type $G_2$}
We use $\mathcal T_{k, l}^{(s)}$ to denote the simple finite-dimensional $U_{q}\widehat{\mathfrak{g}}$-module with highest $l$-weight $T_{k, l}^{(s)}$. Here $T_{k, l}^{(s)}$ is defined in Section \ref{definition of minimal affinizations}. Let $[\mathcal{T}]$ be the equivalence class of the $U_q \widehat{\mathfrak{g}}$-module $\mathcal{T}$ in the Grothendieck ring $\rep( U_q \widehat{\mathfrak{g}} )$.

\begin{theorem}\label{first part of the M system of type G2}
For $s \in \mathbb{Z}$, we have the following system of equations:
\begin{eqnarray}
[\mathcal T_{k, l}^{(s)}][\mathcal T_{k, 0}^{(s+6)}]&=&[\mathcal T_{k+1, 0}^{(s)}][\mathcal T_{k-1, l}^{(s+6)}]+[\mathcal T_{0, 3k+l}^{(s)}] \quad (k \in \mathbb{Z}_{\geq 1}, l \in \{1, 2, 3\}), \label{eqn1}\\
{[\mathcal T_{k, l+3}^{(s)}]}[\mathcal T_{k, l}^{(s+6)}]&=&[\mathcal T_{k+1, l}^{(s)}][\mathcal T_{k-1, l+3}^{(s+6)}]+[\mathcal T_{0, 3k+l+3}^{(s)}][\mathcal T_{0, l}^{(s+6k+6)}] \quad (k, l \in \mathbb{Z}_{\geq 1}) \label{eqn2}.
\end{eqnarray}
Moreover, every module in the summands on the right hand side of the above equations corresponds to simple modules.
\end{theorem}
This is the first part of the M-system of type $G_2$. The equations in Theorem \ref{first part of the M system of type G2} will be proved in Section \ref{proof main1} and the simplicity of the modules in the summands on the right hand side of the equations in Theorem \ref{first part of the M system of type G2} will be proved in Section \ref{proof simplicity of the modules on the right hand side}.

\begin{theorem}\label{second part of M system}
For $s \in \mathbb{Z}$, we have the following system of equations:
\begin{eqnarray*}
[\widetilde{\mathcal T}_{k, l}^{(s)}][\widetilde{\mathcal T}_{k, 0}^{(s+6)}]&=&[\widetilde{\mathcal T}_{k+1, 0}^{(s)}][\widetilde{\mathcal T}_{k-1, l}^{(s+6)}]+[\widetilde{\mathcal T}_{0, 3k+l}^{(s)}] \quad (k \in \mathbb{Z}_{\geq 1}, l \in \{1, 2, 3\}), \\
{[\widetilde{\mathcal T}_{k, l+3}^{(s)}]}[\widetilde{\mathcal T}_{k, l}^{(s+6)}]&=&[\widetilde{\mathcal T}_{k+1, l}^{(s)}][\widetilde{\mathcal T}_{k-1, l+3}^{(s+6)}]+[\widetilde{\mathcal T}_{0, l}^{(s+6k+6)}][\widetilde{\mathcal T}_{0, 3k+l+3}^{(s)}] \quad (k, l\in \mathbb{Z}_{\geq 1}).
\end{eqnarray*}
Moreover, every module in the summands on the right hand side of the above equations corresponds to simple modules.
\end{theorem}
This is the second part of the M-system of type $G_2$. Theorem \ref{second part of M system} will be proved in Section \ref{sec: second part of M system}.

The M-system gives more efficient recursive procedure for computing the $q$-characters of minimal affinizations than the extended T-systems from \cite{LM13}.

The equations in Theorem \ref{first part of the M system of type G2} are equivalent to the following equations.
\begin{equation*}
\begin{split}
&\chi_q(\mathcal T_{k, l}^{(s)}) \chi_q(\mathcal T_{k, 0}^{(s+6)})=\chi_q(\mathcal T_{k+1, 0}^{(s)}) \chi_q(\mathcal T_{k-1, l}^{(s+6)})+\chi_q(\mathcal T_{0, 3k+l}^{(s)}) \quad (k \in \mathbb{Z}_{\geq 1}, l \in \{1, 2, 3\}), \\
&\chi_q(\mathcal T_{k, l+3}^{(s)} ) \chi_q( \mathcal T_{k, l}^{(s+6)})=\chi_q(\mathcal T_{k+1, l}^{(s)}) \chi_q(\mathcal T_{k-1, l+3}^{(s+6)})+\chi_q(\mathcal T_{0, 3k+l+3}^{(s)}) \chi_q(\mathcal T_{0, l}^{(s+6k+6)}) \quad (k, l \in \mathbb{Z}_{\geq 1}).
\end{split}
\end{equation*}

\begin{example}\label{eample1}
The following are some examples of equations in the M-system of type $G_2$.
\begin{eqnarray*}
\scalemath{0.65}{
\begin{split}
[1_{-7} 2_{0}][1_{-1}] & = [1_{-7} 1_{-1}][2_{0}]+[2_{-6} 2_{-4} 2_{-2} 2_{0}], \\
[1_{-9} 2_{-2} 2_{0}] [1_{-3}] & = [1_{-9} 1_{-3}][2_{-2} 2_{0}]+[2_{-8} 2_{-6} 2_{-4} 2_{-2} 2_{0}], \\
[1_{-11} 2_{-4} 2_{-2} 2_{0}] [1_{-5}] & = [1_{-11} 1_{-5}][2_{-4} 2_{-2} 2_{0}]+[2_{-10} 2_{-8} 2_{-6} 2_{-4} 2_{-2} 2_{0}], \\
[1_{-13} 2_{-6} 2_{-4} 2_{-2} 2_{0}] [1_{-7} 2_{0}] & = [1_{-13} 1_{-7} 2_{0}][2_{-6} 2_{-4} 2_{-2} 2_{0}]+[2_{0}][2_{-12} 2_{-10} \cdots 2_{-2} 2_{0}], \\
[1_{-33} 1_{-27} 2_{-20} \cdots 2_{-2} 2_{0}] [1_{-27} 1_{-21} 2_{-14} \cdots 2_{-2} 2_{0}] &= [1_{-33} 1_{-27} 1_{-21} 2_{-14} \cdots 2_{-2} 2_{0}][1_{-27} 2_{-20} \cdots 2_{-2} 2_{0}]+[2_{-14} 2_{-12} \cdots 2_{-2} 2_{0}] [2_{-32} 2_{-30} \cdots 2_{-2} 2_{0}].
\end{split}
}
\end{eqnarray*}
\end{example}

\begin{example}
The following are some examples of equations in the second part of the M-system of type $G_2$.
\begin{eqnarray*}
\scalemath{0.82}{
\begin{split}
[2_{0} 1_{7}] [1_{1}] & = [1_{1} 1_{7}][2_{0}]+[2_{0} 2_{2} 2_{4} 2_{6}], \\
[2_{0} 2_{2} 1_{9}] [1_{3}] & = [1_{3} 1_{9}][2_{0} 2_{2}]+[2_{0} 2_{2} 2_{4} 2_{6} 2_{8}], \\
[2_{0} 2_{2} 2_{4} 1_{11}] [1_{5}] & = [1_{5} 1_{11}][2_{0} 2_{2} 2_{4}]+[2_{0} 2_{2} 2_{4} 2_{6} 2_{8} 2_{10}], \\
[2_{0} 2_{2} 2_{4} 2_{6} 1_{13}] [2_{0} 1_{7}] & = [2_{0} 1_{7} 1_{13}][2_{0} 2_{2} 2_{4} 2_{6}]+[2_{0}][2_{0} 2_{2} \cdots 2_{10} 2_{12}], \\
[2_{0} 2_{2} \cdots 2_{20} 1_{27} 1_{33}] [2_{0} 2_{2} \cdots 2_{14} 1_{21} 1_{27}] & = [2_{0} 2_{2} \cdots 2_{14} 1_{21} 1_{27} 1_{33}][2_{0} 2_{2} \cdots 2_{20} 1_{27}]+[2_{0} 2_{2} \cdots 2_{12} 2_{14}][2_{0} 2_{2} \cdots 2_{30} 2_{32}].
\end{split}
}
\end{eqnarray*}
\end{example}

\subsection{The $m$-system of type $G_2$}
For $k, l \in \mathbb{Z}_{\geq 0}$, let $m_{k,l} = \res(\mathcal{T}_{k,l}^{(0)})$ (resp. $\widetilde{m}{_{k,l}} = \res(\widetilde{\mathcal{T}}{_{k,l}^{(0)}})$) be the restriction of $\mathcal{T}_{k,l}^{(0)}$ (resp. $\widetilde{\mathcal{T}}{_{k,l}^{(0)}}$) to $U_q \mathfrak{g}$. Let $\chi(M)$ be the character of a $U_q \mathfrak{g}$-module $M$. By Theorem \ref{first part of the M system of type G2}, we have the following result.
\begin{corollary}\label{m system}
We have
\begin{eqnarray*}
\chi(m_{k, l}) \chi( m_{k, 0}) &=& \chi(m_{k+1, 0}) \chi( m_{k-1, l}) + \chi( m_{0, 3k+l}) \quad (k \in \mathbb{Z}_{\geq 1}, l \in \{1, 2, 3\}), \\
\chi( m_{k, l+3} ) \chi( m_{k, l} ) &=&  \chi( m_{k+1, l} ) \chi( m_{k-1, l+3} ) + \chi( m_{0, l} ) \chi( m_{0, 3k+l+3} ) \quad (k, l \in \mathbb{Z}_{\geq 1}), \\
\chi(\widetilde{m}{_{k, l}}) \chi( \widetilde{m}{_{k, 0}}) &=& \chi(\widetilde{m}{_{k+1, 0}}) \chi( \widetilde{m}{_{k-1, l}}) + \chi( \widetilde{m}{_{0, 3k+l}}) \quad (k \in \mathbb{Z}_{\geq 1}, l \in \{1, 2, 3\}), \\
\chi( \widetilde{m}{_{k, l+3} }) \chi(\widetilde{ m}{_{k, l}} ) &=&  \chi( \widetilde{m}{_{k+1, l}} ) \chi( \widetilde{m}{_{k-1, l+3}} ) + \chi( \widetilde{m}{_{0, l}} ) \chi( \widetilde{m}{_{0, 3k+l+3}} ) \quad (k, l \in \mathbb{Z}_{\geq 1}).
\end{eqnarray*}
\end{corollary}

We call the above system of equations the $m$-system of type $G_2$.

\section{Interpretation of the equations in the M-system of type $G_2$ as exchange relations} \label{interpret M systems as exchange relations}
In this section, we interpret the equations in the M-system of type $G_2$ as exchange relations in certain cluster algebra constructed by Hernandez and Leclerc in \cite{HL16}.

\subsection{The cluster algebra $\mathscr{A}$ constructed by Hernandez and Leclerc in \cite{HL16}} \label{definition of cluster algebra A}
Let $S=\{-2n+1\mid n\in \mathbb{Z}_{\geq 1}\}$, $S'=\{-2n+2 \mid n\in \mathbb{Z}_{\geq 1}\}$, and $V=(\{1\}\times S)\bigcup (\{2\}\times S')$. Let $Q$ be a quiver with the vertex set ${V}$. The arrows of $Q$ are given by the following rules. For $s_1, s_{2} \in S, s'_{1}, s'{_2} \in S'$, there is an arrow from $(1, s_1)$ to $(1, s_2)$ if and only if $s_2 = s_1 +6$, there is an arrow from $(2, s'_1)$ to $(2, s'_2 )$ if and only if $s'_2 = s'_1 + 2 $, there is an arrow from $(1, s_1)$ to $(2, s'_1)$ if and only if $s'_1 = s_1 -5$, and there is an arrow from $(2, s'_2)$ to $(1, s_2 )$ if and only if $s_2 = s'_2 - 1 $. The quiver $Q$ is the quiver $G^-$ of type $G_2$ in \cite{HL16}.

Let ${\bf t} =\{t_{k, 0}^{(s_1)}, t_{0, l}^{(s_2)} \mid s_1, s_2 \in S, k, l\in \mathbb{Z}_{\geq 1}\}$. Let $\mathscr{A}$ be the cluster algebra defined by the initial seed $({\bf t}, Q)$. By Definition \ref{definition of cluster algebras of infinite rank}, $\mathscr{A}$ is the $\mathbb{Q}$-subalgebra of the field of rational functions $\mathbb{Q}({\bf t})$ generated by all the elements obtained from some elements of $\bf t$ via a finite sequence of seed mutations.

\subsection{Interpretation of the first part of the M-system of type $G_2$ as exchange relations}
We use $``C_{1}"$ to denote the column of vertices $(1, -1)$, $(1, -7)$, $\ldots$, $(1, -6n+5)$, $\cdots$ in the quiver $Q$. We use $``C_{2}"$ to denote the column of vertices $(1, -3)$, $(1, -9)$, $\ldots$, $(1, -6n+3)$, $\cdots$ in $Q$. We use $``C_{3}"$ to denote the column of vertices $(1, -5)$, $(1, -11)$, $\ldots$, $(1, -6n+1)$, $\cdots$ in $Q$. We use $``C_{4}"$ to denote the column of vertices $(2, 0)$, $(2, -2)$, $\ldots$, $(1, -2n+2)$, $\cdots$ in $Q$.

By saying that mutate at the column $C_i$, $i \in \{1, 2, 3, 4\}$, we mean that we mutate the vertices of $C_i$ as follows. First we mutate at the first vertex in the column $C_i$, then the second vertex, an so on until the vertex at infinity. By saying that we mutate $C_{i_1}, C_{i_2}, \ldots$, where $i_j \in \{1, 2, 3, 4\}, j=1,2,\ldots, n$, we mean that we first mutate the column $C_{i_1}$, then the column $C_{i_2}$, an so on.

The variables $t_{k,0}^{(s_1)}$, $t_{0,l}^{(s_2)}$, $s_1, s_2 \in S$, are the cluster variables in the initial seed of $\mathscr{A}$ defined in Section \ref{definition of cluster algebra A}. For convenience, we write $t_{\lceil -s_1/6 \rceil, 0}^{(s_1)}$ at the vertex $(1,s_1)$ and write $t_{0, (-s_2+1)/2}^{(s_2)}$ at the vertex $(2,s_2)$ in the initial quiver $Q$, $s_1 , s_2 \in S$. Then we obtain the quiver (a) in Figure \ref{mutation sequence c1 c1 c1}.

We define some variables $t_{k,l}^{(s)}$ ( $k, l \in \mathbb{Z}_{\geq 1}$, $s \in S$ ) recursively as follows. Let $\text{Seq}_i$, $i=1,2,3$, be the mutation sequence $C_i, C_i, C_i, \ldots$.

We define
\begin{gather} \label{variables 1}
\begin{split}
& t_{k, l}^{(s)} = (t_{k, 0}^{(s+6)})', \quad k \in \mathbb{Z}_{\geq 1}, l \in \{1, 2, 3\}, s \equiv 2l+3 \pmod 6, \\
& t_{k, l+3}^{(s)} = (t_{k, l}^{(s+6)})', \quad k, l \in \mathbb{Z}_{\geq 1},
\end{split}
\end{gather}
where
\begin{gather} \label{exchange relations 1}
\begin{split}
& (t_{k, 0}^{(s+6)})' = \frac{t_{k+1, 0}^{(s)}t_{k-1, l}^{(s+6)}+t_{0, 3k+l}^{(s)} }{t_{k, 0}^{(s+6)}}, \quad k \in \mathbb{Z}_{\geq 1}, l \in \{1, 2, 3\}, \\
& (t_{k, l}^{(s+6)})' = \frac{t_{k+1, l}^{(s)}t_{k-1, l+3}^{(s+6)}+t_{0, 3k+l+3}^{(s)} t_{0, l}^{(s+6k+6)}}{t_{k, l}^{(s+6)}}, \quad k, l \in \mathbb{Z}_{\geq 1}.
\end{split}
\end{gather}
are exchange relations which occur when we mutate $\text{Seq}_i$, $i \in \{1,2,3\}$. The variables (\ref{variables 1}) are defined in the order according to the mutation sequence $\text{Seq}_i$. In this order, every variable in (\ref{variables 1}) is defined by an equation of (\ref{exchange relations 1}) using variables in ${\bf t}$ and those variables in (\ref{variables 1}) which are already defined.

Figure \ref{mutation sequence c1 c1 c1} is the first few mutations in the mutation sequence $\text{Seq}_1$.

The exchange relations in (\ref{exchange relations 1}) coincides with the equations in the first part of the M-system of type $G_2$ in Theorem \ref{first part of the M system of type G2}. Therefore the equations in the first part of the M-system of type $G_2$ can be interpreted as exchange relations in the cluster algebra $\mathscr{A}$. The cluster variables $t_{k,l}^{(s)}$ corresponds to the minimal affinizations $\mathcal{T}_{k,l}^{(s)}$, $k, l \in \mathbb{Z}_{\geq 0}$.

Using the mutation sequence $\text{Seq}_i$, $i \in \{1,2,3\}$, we obtain minimal affinizations
\begin{align*}
\mathcal{T}_{k, l}^{(-6k-2l+1 )}, \ k, l \in \mathbb{Z}_{\geq 1}, \ l \equiv i \pmod 3.
\end{align*}

\subsection{Interpretation of the second part of the M-system of type $G_2$ as exchange relations}
We can also interpret the second part of the M-system of type $G_2$ as exchange relations in the cluster algebra $\mathscr{A}$ defined in Section \ref{definition of cluster algebra A}. Let $\text{Seq}_i$, $i=1,2,3$, be the mutation sequence $C_i, C_i, C_i, \ldots$. The cluster variables $t_{k,l}^{(s)}$ corresponds to the minimal affinizations $\widetilde{\mathcal{T}}_{k,l}^{(s)}$, $k, l \in \mathbb{Z}_{\geq 0}$. Using the mutation sequence $\text{Seq}_i$, $i \in \{1,2,3\}$, we obtain minimal affinizations
\begin{align*}
\widetilde{\mathcal{T}}_{k, l}^{(-6k-2l+1 )}, \ k, l \in \mathbb{Z}_{\geq 1}, \ l \equiv i \pmod 3.
\end{align*}

\begin{figure}
\begin{minipage}[t]{.25\linewidth}
\begin{xy}
(13,90) *+{t_{0, 1}^{(-1)}} ="4";%
(13,80) *+{t_{0, 2}^{(-3)}} ="5";(26,80) *+{t_{1, 0}^{(-1)}} ="13";%
(0,70) *+{t_{1, 0}^{(-3)}} ="1";(13,70) *+{t_{0, 3}^{(-5)}} ="6";%
(13,60) *+{t_{0, 4}^{(-7)}} ="7";(39,60) *+{t_{1, 0}^{(-5)}} ="16";%
(13,50) *+{t_{0, 5}^{(-9)}} ="8";(26,50) *+{t_{2, 0}^{(-7)}} ="14";%
(0,40) *+{t_{2, 0}^{(-9)}} ="2";(13,40) *+{t_{0, 6}^{(-11)}} ="9";%
(13,30) *+{t_{0, 7}^{(-13)}} ="10";(39,30) *+{t_{2, 0}^{(-11)}} ="17";%
(13,20) *+{t_{0, 8}^{(-15)}} ="11";(26,20) *+{t_{3, 0}^{(-13)}} ="15";%
(0,10) *+{t_{3, 0}^{(-15)}} ="3";(13,10) *+{t_{0, 9}^{(-17)}} ="12";%
(0,0) *+{\vdots} ="18";(13,0) *+{\vdots} ="19";(20,-5) *+{(\text{a})} ="22";(26,0) *+{\vdots} ="20";(39,0) *+{\vdots} ="21";
{\ar "5";"4"};{\ar "6";"5"};{\ar "7";"6"};{\ar "8";"7"};{\ar "9";"8"};{\ar "10";"9"};{\ar "11";"10"};{\ar "12";"11"};{\ar "19";"12"};%
{\ar "2";"1"};{\ar "3";"2"};{\ar "18";"3"};%
{\ar "14";"13"};{\ar "15";"14"};{\ar "20";"15"};%
{\ar "17";"16"};{\ar "21";"17"};%
{\ar "5";"1"};{\ar "1";"8"};{\ar "8";"2"};{\ar "2";"11"};%
{\ar "4";"13"};{\ar "13";"7"};{\ar "7";"14"};{\ar "14";"10"};{\ar "10";"15"};%
{\ar "6";"16"};{\ar "16";"9"};{\ar "9";"17"};{\ar "17";"12"};
\end{xy}
\end{minipage}
\begin{minipage}[t]{.25\linewidth}
\begin{xy}
(13,90) *+{t_{0, 1}^{(-1)}} ="4";%
(13,80) *+{t_{0, 2}^{(-3)}} ="5";(26,80) *+{\fbox {$t_{1, 1}^{(-7)}$}} ="13";%
(0,70) *+{t_{1, 0}^{(-3)}} ="1";(13,70) *+{t_{0, 3}^{(-5)}} ="6";%
(13,60) *+{t_{0, 4}^{(-7)}} ="7";(39,60) *+{t_{1, 0}^{(-5)}} ="16";%
(13,50) *+{t_{0, 5}^{(-9)}} ="8";(26,50) *+{t_{2, 0}^{(-7)}} ="14";%
(0,40) *+{t_{2, 0}^{(-9)}} ="2";(13,40) *+{t_{0, 6}^{(-11)}} ="9";%
(13,30) *+{t_{0, 7}^{(-13)}} ="10";(39,30) *+{t_{2, 0}^{(-11)}} ="17";%
(13,20) *+{t_{0, 8}^{(-15)}} ="11";(26,20) *+{t_{3, 0}^{(-13)}} ="15";%
(0,10) *+{t_{3, 0}^{(-15)}} ="3";(13,10) *+{t_{0, 9}^{(-17)}} ="12";%
(0,0) *+{\vdots} ="18";(13,0) *+{\vdots} ="19";(20,-5) *+{(\text{b})} ="22";(26,0) *+{\vdots} ="20";(39,0) *+{\vdots} ="21";
{\ar "5";"4"};{\ar "6";"5"};{\ar "7";"6"};{\ar "8";"7"};{\ar "9";"8"};{\ar "10";"9"};{\ar "11";"10"};{\ar "12";"11"};{\ar "19";"12"};%
{\ar "2";"1"};{\ar "3";"2"};{\ar "18";"3"};%
{\ar "13";"14"};{\ar "15";"14"};{\ar "20";"15"};%
{\ar "17";"16"};{\ar "21";"17"};%
{\ar "5";"1"};{\ar "1";"8"};{\ar "8";"2"};{\ar "2";"11"};%
{\ar "13";"4"};{\ar "7";"13"};{\ar "14";"10"};{\ar "10";"15"};%
{\ar "6";"16"};{\ar "16";"9"};{\ar "9";"17"};{\ar "17";"12"};%
{\ar@/_2pc/"4";"7"};
\end{xy}
\end{minipage}
\begin{minipage}[t]{.25\linewidth}
\begin{xy}
(13,90) *+{t_{0, 1}^{(-1)}} ="4";%
(13,80) *+{t_{0, 2}^{(-3)}} ="5";(26,80) *+{t_{1, 1}^{(-7)}} ="13";%
(0,70) *+{t_{1, 0}^{(-3)}} ="1";(13,70) *+{t_{0, 3}^{(-5)}} ="6";%
(13,60) *+{t_{0, 4}^{(-7)}} ="7";(39,60) *+{t_{1, 0}^{(-5)}} ="16";%
(13,50) *+{t_{0, 5}^{(-9)}} ="8";(26,50) *+{\fbox {$t_{2, 1}^{(-13)}$}} ="14";%
(0,40) *+{t_{2, 0}^{(-9)}} ="2";(13,40) *+{t_{0, 6}^{(-11)}} ="9";%
(13,30) *+{t_{0, 7}^{(-13)}} ="10";(39,30) *+{t_{2, 0}^{(-11)}} ="17";%
(13,20) *+{t_{0, 8}^{(-15)}} ="11";(26,20) *+{t_{3, 0}^{(-13)}} ="15";%
(0,10) *+{t_{3, 0}^{(-15)}} ="3";(13,10) *+{t_{0, 9}^{(-17)}} ="12";%
(0,0) *+{\vdots} ="18";(13,0) *+{\vdots} ="19";(20,-5) *+{(\text{c})} ="22";(26,0) *+{\vdots} ="20";(39,0) *+{\vdots} ="21";
{\ar "5";"4"};{\ar "6";"5"};{\ar "7";"6"};{\ar "8";"7"};{\ar "9";"8"};{\ar "10";"9"};{\ar "11";"10"};{\ar "12";"11"};{\ar "19";"12"};%
{\ar "2";"1"};{\ar "3";"2"};{\ar "18";"3"};%
{\ar "14";"13"};{\ar "14";"15"};{\ar "20";"15"};%
{\ar "17";"16"};{\ar "21";"17"};%
{\ar "5";"1"};{\ar "1";"8"};{\ar "8";"2"};{\ar "2";"11"};%
{\ar "13";"4"};{\ar "13";"10"};{\ar "7";"13"};{\ar "10";"14"};%
{\ar "6";"16"};{\ar "16";"9"};{\ar "9";"17"};{\ar "17";"12"};%
{\ar@/_2pc/"4";"7"};
\end{xy}
\end{minipage}

\begin{minipage}[t]{.10\linewidth}
\begin{xy}
(11,90) *+{t_{0, 1}^{(-1)}} ="4";%
(11,80) *+{t_{0, 2}^{(-3)}} ="5";(22,80) *+{t_{1, 1}^{(-7)}} ="13";%
(0,70) *+{t_{1, 0}^{(-3)}} ="1";(11,70) *+{t_{0, 3}^{(-5)}} ="6";%
(11,60) *+{t_{0, 4}^{(-7)}} ="7";(33,60) *+{t_{1, 0}^{(-5)}} ="16";%
(11,50) *+{t_{0, 5}^{(-9)}} ="8";(22,50) *+{t_{2, 1}^{(-13)}} ="14";%
(0,40) *+{t_{2, 0}^{(-9)}} ="2";(11,40) *+{t_{0, 6}^{(-11)}} ="9";%
(11,30) *+{t_{0, 7}^{(-13)}} ="10";(33,30) *+{t_{2, 0}^{(-11)}} ="17";%
(11,20) *+{t_{0, 8}^{(-15)}} ="11";(22,20) *+{\fbox {$t_{3, 1}^{(-19)}$}} ="15";%
(0,10) *+{t_{3, 0}^{(-15)}} ="3";(11,10) *+{t_{0, 9}^{(-17)}} ="12";%
(0,0) *+{\vdots} ="18";(11,0) *+{\vdots} ="19";(15,-5) *+{(\text{d})} ="22";(22,0) *+{\vdots} ="20";(33,0) *+{\vdots} ="21";
{\ar "5";"4"};{\ar "6";"5"};{\ar "7";"6"};{\ar "8";"7"};{\ar "9";"8"};{\ar "10";"9"};{\ar "11";"10"};{\ar "12";"11"};{\ar "19";"12"};%
{\ar "2";"1"};{\ar "3";"2"};{\ar "18";"3"};%
{\ar "14";"13"};{\ar "14";"15"};{\ar "20";"15"};%
{\ar "17";"16"};{\ar "21";"17"};%
{\ar "5";"1"};{\ar "1";"8"};{\ar "8";"2"};{\ar "2";"11"};%
{\ar "13";"4"};{\ar "13";"10"};{\ar "7";"13"};{\ar "10";"14"};%
{\ar "6";"16"};{\ar "16";"9"};{\ar "9";"17"};{\ar "17";"12"};%
{\ar@/_2pc/"4";"7"};
\end{xy}
\end{minipage}
\begin{minipage}[t]{.10\linewidth}
\begin{xy}
(2,60) *+{\cdots} ="4";
\end{xy}
\end{minipage}
\begin{minipage}[t]{.10\linewidth}
\begin{xy}
(11,92) *+{t_{0, 1}^{(-1)}} ="4";%
(11,80) *+{t_{0, 2}^{(-3)}} ="5";(22,80) *+{\fbox {$t_{1, 4}^{(-13)}$}} ="13";%
(0,70) *+{t_{1, 0}^{(-3)}} ="1";(11,70) *+{t_{0, 3}^{(-5)}} ="6";%
(11,60) *+{t_{0, 4}^{(-7)}} ="7";(33,60) *+{t_{1, 0}^{(-5)}} ="16";%
(11,50) *+{t_{0, 5}^{(-9)}} ="8";(22,50) *+{t_{2, 1}^{(-13)}} ="14";%
(0,40) *+{t_{2, 0}^{(-9)}} ="2";(11,40) *+{t_{0, 6}^{(-11)}} ="9";%
(11,30) *+{t_{0, 7}^{(-13)}} ="10";(33,30) *+{t_{2, 0}^{(-11)}} ="17";%
(11,20) *+{t_{0, 8}^{(-15)}} ="11";(22,20) *+{t_{3, 1}^{(-19)}} ="15";%
(0,10) *+{t_{3, 0}^{(-15)}} ="3";(11,10) *+{t_{0, 9}^{(-17)}} ="12";%
(0,0) *+{\vdots} ="18";(11,0) *+{\vdots} ="19";(15,-5) *+{(\text{e})} ="22";(22,0) *+{\vdots} ="20";(33,0) *+{\vdots} ="21";
{\ar "5";"4"};{\ar "6";"5"};{\ar "7";"6"};{\ar "8";"7"};{\ar "9";"8"};{\ar "10";"9"};{\ar "11";"10"};{\ar "12";"11"};{\ar "19";"12"};%
{\ar "2";"1"};{\ar "3";"2"};{\ar "18";"3"};%
{\ar "13";"14"};{\ar "14";"15"};{\ar "20";"15"};%
{\ar "17";"16"};{\ar "21";"17"};%
{\ar "5";"1"};{\ar "1";"8"};{\ar "8";"2"};{\ar "2";"11"};%
{\ar "4";"13"};{\ar "10";"13"};{\ar "13";"7"};{\ar "10";"14"};%
{\ar "6";"16"};{\ar "16";"9"};{\ar "9";"17"};{\ar "17";"12"};%
{\ar@/_2pc/"7";"10"};
\end{xy}
\end{minipage}
\begin{minipage}[t]{.10\linewidth}
\begin{xy}
(11,90) *+{t_{0, 1}^{(-1)}} ="4";%
(11,80) *+{t_{0, 2}^{(-3)}} ="5";(22,80) *+{t_{1, 4}^{(-13)}} ="13";%
(0,70) *+{t_{1, 0}^{(-3)}} ="1";(11,70) *+{t_{0, 3}^{(-5)}} ="6";%
(11,60) *+{t_{0, 4}^{(-7)}} ="7";(33,60) *+{t_{1, 0}^{(-5)}} ="16";%
(11,50) *+{t_{0, 5}^{(-9)}} ="8";(22,50) *+{\fbox {$t_{2, 4}^{(-19)}$}} ="14";%
(0,40) *+{t_{2, 0}^{(-9)}} ="2";(11,40) *+{t_{0, 6}^{(-11)}} ="9";%
(11,30) *+{t_{0, 7}^{(-13)}} ="10";(33,30) *+{t_{2, 0}^{(-11)}} ="17";%
(11,20) *+{t_{0, 8}^{(-15)}} ="11";(22,20) *+{t_{3, 1}^{(-19)}} ="15";%
(0,10) *+{t_{3, 0}^{(-15)}} ="3";(11,10) *+{t_{0, 9}^{(-17)}} ="12";%
(0,0) *+{\vdots} ="18";(11,0) *+{\vdots} ="19";(15,-5) *+{(\text{f})} ="22";(22,0) *+{\vdots} ="20";(33,0) *+{\vdots} ="21";
{\ar "5";"4"};{\ar "6";"5"};{\ar "7";"6"};{\ar "8";"7"};{\ar "9";"8"};{\ar "10";"9"};{\ar "11";"10"};{\ar "12";"11"};{\ar "19";"12"};%
{\ar "2";"1"};{\ar "3";"2"};{\ar "18";"3"};%
{\ar "13";"14"};{\ar "14";"15"};{\ar "20";"15"};%
{\ar "17";"16"};{\ar "21";"17"};%
{\ar "5";"1"};{\ar "1";"8"};{\ar "8";"2"};{\ar "2";"11"};%
{\ar "4";"13"};{\ar "10";"13"};{\ar "13";"7"};{\ar "10";"14"};%
{\ar "6";"16"};{\ar "16";"9"};{\ar "9";"17"};{\ar "17";"12"};%
{\ar@/_2pc/"7";"10"};
\end{xy}
\end{minipage}
\begin{minipage}[t]{.10\linewidth}
\begin{xy}
(2,60) *+{\cdots} ="4";
\end{xy}
\end{minipage}
\caption{The mutation sequence $C_1, C_1, C_1, \ldots$}\label{mutation sequence c1 c1 c1}
\end{figure}
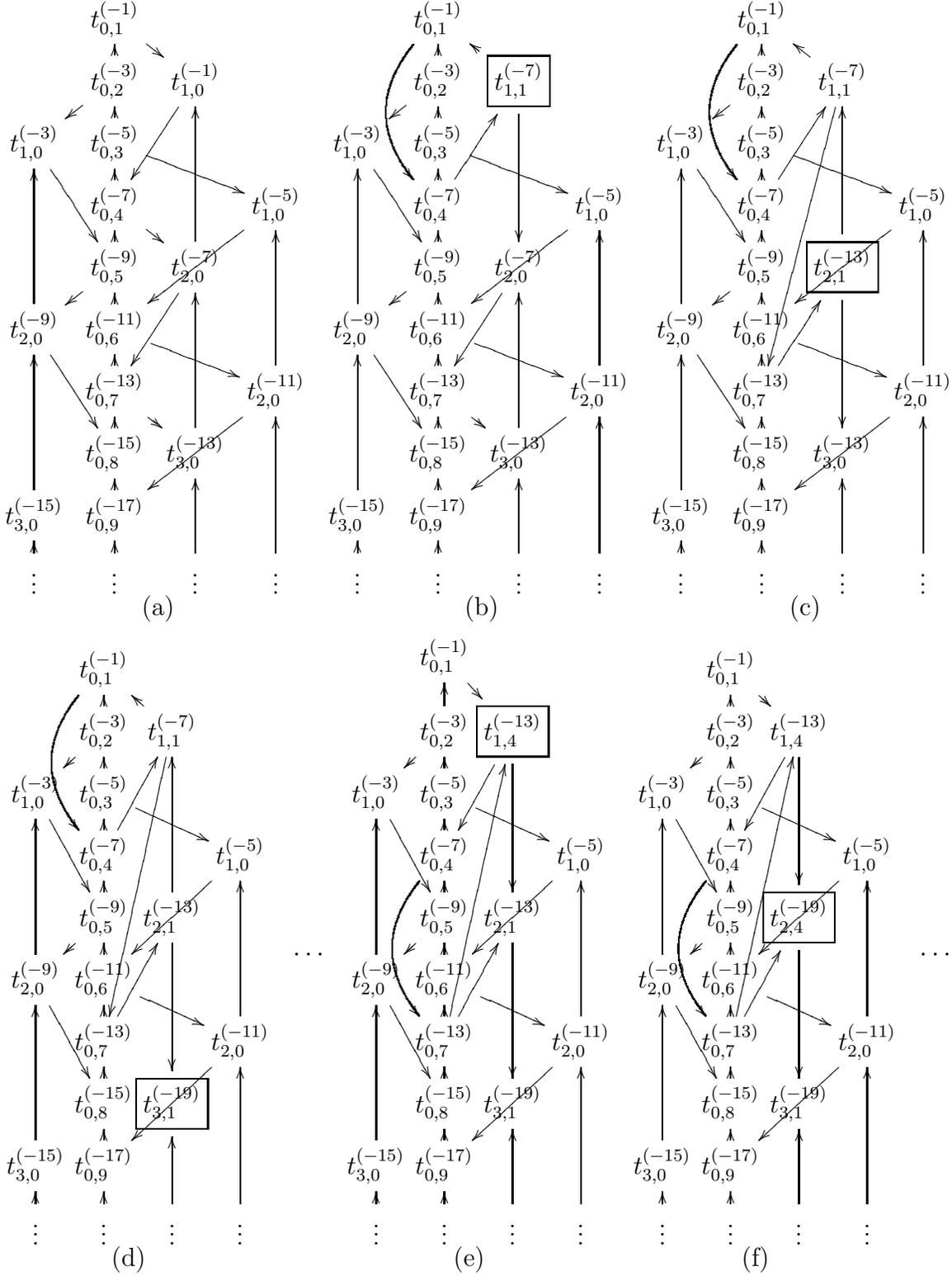

\section{Proof of the equations in Theorem $\ref{first part of the M system of type G2}$} \label{proof main1}
In this section, we prove the equations in Theorem $\ref{first part of the M system of type G2}$.

Using the Frenkel-Mukhin algorithm, one can easily compute the $q$-characters of the fundamental modules.
\begin{lemma}\label{q characters of fundamental modules}
The fundamental q-characters for $U_{q}\widehat{\mathfrak{g}}$ of type $G_{2}$ are given by
\begin{eqnarray*}
\chi_{q}(1_{0})&=&1_{0}+2_{1} 2_{3} 2_{5} 1_{6}^{-1}+2_{1} 2_{3} 2_{7}^{-1}+2_{1} 2_{5}^{-1} 2_{7}^{-1} 1_{4}+2_{3}^{-1} 2_{5}^{-1} 2_{7}^{-1}
1_{2} 1_{4}\\
&&+2_{1} 2_{9} 1_{10}^{-1}+1_{4} 1_{8}^{-1}+2_{3}^{-1} 2_{9} 1_{2} 1_{10}^{-1}+2_{5} 2_{7} 2_{9} 1_{8}^{-1} 1_{10}^{-1}+2_{1} 2_{11}^{-1}\\
&&+2_{3}^{-1} 2_{11}^{-1} 1_{2}+2_{5} 2_{7} 2_{11}^{-1} 1_{8}^{-1}+2_{5} 2_{9}^{-1} 2_{11}^{-1}+2_{7}^{-1} 2_{9}^{-1} 2_{11}^{-1} 1_{6}+1_{12}^{-1}, \\
\chi_{q}(2_{0})&=&2_{0}+2_{2}^{-1} 1_{1}+2_{4} 2_{6} 1_{7}^{-1}+2_{4} 2_{8}^{-1}+2_{6}^{-1} 2_{8}^{-1} 1_{5}+2_{10} 1_{11}^{-1}+2_{12}^{-1}.
\end{eqnarray*}
\end{lemma}

\subsection{Classification of dominant monomials in the summands on both sides of the M-system}
By {Theorem 3.8 in \cite{H07}} (see also {Theorem 3.3 in \cite{LM13}}), the modules $\mathcal T{_{k, l}^{(s)}}$ ($s \in \mathbb{Z}, k, l \in \mathbb{Z}_{\geq 0}$) are special. Therefore we can use the Frenkel-Mukhin algorithm to compute the $q$-characters of $\mathcal T{_{k, l}^{(s)}}$ ($s \in \mathbb{Z}, k, l \in \mathbb{Z}_{\geq 0}$).
Now we use the Frenkel-Mukhin algorithm to classify dominant monomials in the summands on both sides of the M-system.
\begin{lemma}\label{lemma1}
We have the following cases.
\begin{enumerate}[(1)]
\item Let
\begin{align*}
M =T_{k, l}^{(s)} T_{k, 0}^{(s+6)} \ (k\in \mathbb{Z}_{\geq 1}, l \in \{1, 2, 3\}).
\end{align*}
Then the dominant monomials in $\chi_q(T_{k, l}^{(s)}) \chi_q( T_{k, 0}^{(s+6)} )$ $(k\in \mathbb{Z}_{\geq 1}, l \in \{1, 2, 3\})$ are $M$ and
\begin{equation*}
\begin{split}
&M_{i} =M\prod_{j=0}^{i-1}A_{1, s+6k-6j-3}^{-1}, \quad i=1, 2, \ldots, k,
\end{split}
\end{equation*}
with multiplicity $1$.

The dominant monomials in $\chi_q(T_{k+1, 0}^{(s)}) \chi_q(T_{k-1, l}^{(s+6)})$ $(k\in \mathbb{Z}_{\geq 1}, l \in \{1, 2, 3\})$ are $M$, $M_{1}$, $\ldots$, $M_{k-1}$, with multiplicity $1$.

The only dominant monomial in $T_{0, 3k+l}^{(s)}$ $(k\in \mathbb{Z}_{\geq 1}, l \in \{1, 2, 3\})$ is $M_{k}$ with multiplicity $1$.

\item Let
\begin{align*}
M=T_{k, l+3}^{(s)} T_{k, l}^{(s+6)} \ (k, l \in \mathbb{Z}_{\geq 1}).
\end{align*}
Then the dominant monomials in $\chi_q(T_{k, l+3}^{(s)})\chi_q(T _{k, l}^{(s+6)})$ $(k, l \in \mathbb{Z}_{\geq 1})$ are $M$ and
\begin{equation*}
\begin{split}
M_{i} =M\prod_{j=0}^{i-1}A_{1, s+6k-6j-3}^{-1}, \quad i=1, 2, \ldots, k,
\end{split}
\end{equation*}
with multiplicity $1$.

The dominant monomials in $\chi_q(T_{k+1, l}^{(s)}) \chi_q(T_{k-1, l+3}^{(s+6)})$ $(k, l \in \mathbb{Z}_{\geq 1})$ are $M$, $M_{1}$, $\ldots$, $M_{k-1}$, with multiplicity $1$.

The only dominant monomial in $\chi_q(T_{0, 3k+l+3}^{(s)}) \chi_q(T_{0, l}^{(s+6k+6)})$ $(k, l \in \mathbb{Z}_{\geq 1})$ is $M_{k}$ with multiplicity $1$.
\end{enumerate}
\end{lemma}

\begin{proof}
We will prove part (2). Part (1) is similar. In the proof, we use Lemma \ref{need to expand from right most factors} frequently to show that some monomials cannot occur in a $q$-character.

{\bf Classify the dominant monomials in $\chi_{q}(\mathcal T_{k, l+3}^{(s)})\chi_{q}(\mathcal T_{k, l}^{(s+6)})$.}

Let $m_{1}'=T_{k, l+3}^{(s)}$, $m_{2}'=T_{k, l}^{(s+6)}$. Without loss of generality, we may assume that $s=0$. Then
\begin{equation*}
\begin{split}
&m_{1}'=(1_{0} 1_{6} \cdots 1_{6k-6})(2_{6k+1} 2_{6k+3} \cdots 2_{6k+2l+5}), \\
&m_{2}'=(1_{6} 1_{12} \cdots 1_{6k})(2_{6k+7} 2_{6k+9} \cdots 2_{6k+2l+5}).
\end{split}
\end{equation*}
By Theorem \ref{minimal affinizations in X1 are special}, we can use Frenkel-Mukhin algorithm to compute $\chi_{q}(m_{1}')$ and $\chi_{q}(m_{2}')$.

We want to classify all dominant monomials $m = m_{1}m_{2}$, $m_{i} \in \chi_{q}(m_{i}'), i=1, 2$. Let $m = m_{1}m_{2}$ be a dominant monomial, where $m_{i} \in \chi_{q}(m_{i}'), i=1, 2$. We denote
\begin{align*}
& m_{3} =2_{6k+1} 2_{6k+3} \cdots 2_{6k+2l+5}, \\
& m_{4} =2_{6k+7} 2_{6k+9} \cdots 2_{6k+2l+5}.
\end{align*}

We have the following cases.

{\bf Case 1.}
\begin{align*}
& m_{1} \in \chi_q(m_1') \cap  \chi_{q}(1_{0} 1_{6} \cdots 1_{6k-6})(\chi_{q}(m_{3})-m_{3}), \\
& m_{2} \in \chi_q(m_2') \cap \chi_{q}(1_{6} 1_{12} \cdots 1_{6k})(\chi_{q}(m_{4})-m_{4}).
\end{align*}

We have $m_1 = xy$, $x \in \chi_{q}(1_{0} 1_{6} \cdots 1_{6k-6})$, $y \in \chi_{q}(m_{3})-m_{3}$. By Lemma \ref{non highest monomials in q characters of KR modules are right negative}, $y$ is right negative since $L(m_3)$ is a Kirillov-Reshetikhin module. If $x = 1_{0} 1_{6} \cdots 1_{6k-6}$, then $m_1=xy$ must be right negative because the largest index in $x$ is $6k-6$ and $1_{6k-6}$ cannot cancel the negative factors in $y$ (all indices of the factors in $y$ are larger than $6k-6$). If $x \in \chi_{q}(1_{0} 1_{6} \cdots 1_{6k-6}) - 1_{0} 1_{6} \cdots 1_{6k-6}$, then $x$ is right negative since $L(1_{0} 1_{6} \cdots 1_{6k-6})$ is a Kirillov-Reshetikhin module. By Lemma \ref{properties of right negative monomials}, the product of two right negative monomials are right negative. Therefore $m_1=xy$ is right negative.

Similarly, $m_2$ is right negative. It follows that $m = m_{1}m_{2}$ is right negative and hence $m$ is not dominant. This contradicts our assumption.

{\bf Case 2.}
\begin{align*}
& m_{1} \in \chi_q(m_1') \cap  \chi_{q}(1_{0} 1_{6} \cdots 1_{6k-6})(\chi_{q}(m_{3})-m_{3}), \\
& m_{2} \in \chi_q(m_2') \cap  \chi_{q}(1_{6} 1_{12} \cdots 1_{6k})m_4.
\end{align*}
In this case, the indices of the negative factors in $m_1$ are larger than $6k+2l+5$. By Lemma \ref{possible monomials}, the largest index in $m_2$ is $6k+2l+5$. It follows that the negative factor with largest index in $m_1$ cannot be canceled by the factors in $m_2$. Therefore $m = m_{1}m_{2}$ is right negative and hence $m$ is not dominant. This contradicts our assumption.

{\bf Case 3.}
\begin{align*}
& m_{1} \in \chi_q(m_1') \cap  \chi_{q}(1_{0} 1_{6} \cdots 1_{6k-6})m_3, \\
& m_{2} \in \chi_q(m_2') \cap  \chi_{q}(1_{6} 1_{12} \cdots 1_{6k})(\chi_{q}(m_{4})-m_{4}).
\end{align*}
By using the same argument as Case 2, we have that $m=m_1 m_2$ is right negative and hence $m$ is not dominant. This contradicts our assumption.

{\bf Case 4.}
\begin{align*}
& m_{1} \in \chi_q(m_1') \cap  \chi_{q}(1_{0} 1_{6} \cdots 1_{6k-6})m_3, \\
& m_{2} \in \chi_q(m_2') \cap  \chi_{q}(1_{6} 1_{12} \cdots 1_{6k})m_4.
\end{align*}

We need the following lemma.
\begin{lemma} \label{possible monomials}
\begin{enumerate}
\item Suppose that
\begin{align*}
& m_{1} \in \chi_q(m_1') \cap \chi_{q}(1_{0} 1_{6} \cdots 1_{6k-6})m_3.
\end{align*}
Then $m_1$ is one of the following monomials:
\begin{align*}
& m_1', \\
& n_{1} = m_{1}'A_{1, 6k-3}^{-1} = 1_{0} 1_{6} \cdots 1_{6k-12} 1_{6k}^{-1} 2_{6k-5} 2_{6k-3} \cdots 2_{6k+2l+5}, \\
& n_{2} = m_{1}'A_{1, 6k-3}^{-1}A_{1, 6k-9}^{-1} = 1_{0} 1_{6} \cdots 1_{6k-18} 1_{6k-6}^{-1} 1_{6k}^{-1} 2_{6k-11} 2_{6k-9} \cdots 2_{6k+2l+5}, \\
& \cdots \\
& n_{k} = m_{1}'A_{1, 6k-3}^{-1}A_{1, 6k-9}^{-1} \cdots A_{1, 3}^{-1} = 1_{6}^{-1} \cdots 1_{6k-6}^{-1} 1_{6k}^{-1} 2_{1} 2_{3} \cdots 2_{6k+2l+5}.
\end{align*}

\item Suppose that
\begin{align*}
& m_{2} \in \chi_q(m_2') \cap \chi_{q}(1_{6} 1_{12} \cdots 1_{6k})m_4.
\end{align*}
Then $m_2$ is one of the following monomials:
\begin{align*}
& m_2', \\
& m_{2}'A_{1, 6k+3}^{-1} =  1_{6} \cdots 1_{6k-6} 1_{6k+6}^{-1} 2_{6k+1} 2_{6k+3} \cdots 2_{6k+2l+5}, \\
& m_{2}'A_{1, 6k+3}^{-1}A_{1, 6k-3}^{-1} = 1_{6} \cdots 1_{6k-12} 1_{6k}^{-1} 1_{6k+6}^{-1} 2_{6k-5} 2_{6k-3} \cdots 2_{6k+2l+5}, \\
& \cdots \\
& m_{2}'A_{1, 6k+3}^{-1}A_{1, 6k-3}^{-1} \cdots A_{1, 9}^{-1} = 1_{12}^{-1} 1_{18}^{-1} \cdots 1_{6k}^{-1} 1_{6k+6}^{-1} 2_{7} 2_{9} \cdots 2_{6k+2l+5}.
\end{align*}
\end{enumerate}
\end{lemma}

\begin{proof}
We will prove part (1). Part (2) can be proved similarly. Suppose that $m_{1} \in \chi_q(m_1') \cap \chi_{q}(1_{0} 1_{6} \cdots 1_{6k-6})m_3$. We have
\begin{align*}
m_1 \in \chi_q(m_1') \cap \chi_q(1_0 1_6 \cdots 1_{6k-12}) 1_{6k-6} m_3
\end{align*}
or
\begin{align*}
m_1 \in \chi_q(m_1') \cap \chi_q(1_0 1_6 \cdots 1_{6k-12}) (\chi_q(1_{6k-6}) - 1_{6k-6})m_3.
\end{align*}

If $m_1 \in \chi_q(m_1') \cap \chi_q(1_0 1_6 \cdots 1_{6k-12}) 1_{6k-6} m_3$, then $m_1 \in \varphi_1( m_1' )$, where the map $\varphi_1$ is defined before Theorem \ref{Uqsl_2 arguments}. By Lemma \ref{need to expand from right most factors}, we have $m_1 = m_1'$ since $\beta_1( 1_0 1_6 \cdots 1_{6k-12} 1_{6k-6} m_3 ) = 1_0 1_6 \cdots 1_{6k-6}$ is a $q_1$-string in $m_1'$ and $1_{6k-6}$ is a factor of $m_1$.

If $m_1 \in \chi_q(m_1') \cap \chi_q(1_0 1_6 \cdots 1_{6k-12}) (\chi_q(1_{6k-6}) - 1_{6k-6})m_3$, then
\begin{align*}
m_1 \in \chi_q(m_1') \cap \chi_q(1_0 1_6 \cdots 1_{6k-12}) 1_{6k}^{-1} 2_{6k-5} 2_{6k-3} 2_{6k-1} m_3
\end{align*}
since $ 2_{6k-5} 2_{6k-3} 2_{6k-1} m_3 = 2_{6k-5} 2_{6k-3} 2_{6k-1} \cdots 2_{6k+2l+5}$ is a $q_2$-string and $2_{6k+2l+5}$ is a factor of $m_1$.

By the same argument, since $2_{6k-5} 2_{6k-3} 2_{6k-1} \cdots 2_{6k+2l+5}$ is a $q_2$-string and $2_{6k+2l+5}$ is a factor of $m_1$, by Lemma \ref{need to expand from right most factors} we have that $m_1=m_1'$ or
\begin{align*}
& m_{1} = n_{1} = m_{1}'A_{1, 6k-3}^{-1} = 1_{0} 1_{6} \cdots 1_{6k-12} 1_{6k}^{-1} 2_{6k-5} 2_{6k-3} \cdots 2_{6k+2l+5},
\end{align*}
or
\begin{align*}
& m_{1} \in \chi_q(m_1') \cap \chi_{q}(1_{0} 1_{6} \cdots 1_{6k-18}) 1_{6k-6}^{-1} 1_{6k}^{-1} 2_{6k-11} 2_{6k-9} \cdots 2_{6k+2l+5}.
\end{align*}
Using the same argument, we have that $m_{1}$ must be one of the following monomials:
\begin{align*}
& m_1', \\
& n_{1} = m_{1}'A_{1, 6k-3}^{-1} = 1_{0} 1_{6} \cdots 1_{6k-12} 1_{6k}^{-1} 2_{6k-5} 2_{6k-3} \cdots 2_{6k+2l+5}, \\
& n_{2} = m_{1}'A_{1, 6k-3}^{-1}A_{1, 6k-9}^{-1} = 1_{0} 1_{6} \cdots 1_{6k-18} 1_{6k-6}^{-1} 1_{6k}^{-1} 2_{6k-11} 2_{6k-9} \cdots 2_{6k+2l+5}, \\
& \cdots \\
& n_{k} = m_{1}'A_{1, 6k-3}^{-1}A_{1, 6k-9}^{-1} \cdots A_{1, 3}^{-1} = 1_{6}^{-1} \cdots 1_{6k-6}^{-1} 1_{6k}^{-1} 2_{1} 2_{3} \cdots 2_{6k+2l+5}.
\end{align*}
\end{proof}

In this case, we have $m_{1} \in \chi_{q}(1_{0} 1_{6} \cdots 1_{6k-6})m_{3}$ and $m_{2} \in \chi_{q}(1_{6} 1_{12} \cdots 1_{6k})m_{4}$. Since $m=m_1 m_2$ is dominant, by Lemma \ref{possible monomials} we have that $m=m_1 m_2$ is one of the following dominant monomials
\begin{eqnarray*}
\begin{split}
&M= m_{1}'m_{2}', \ M_{1} =n_{1}m_{2}'=MA_{1, 6k-3}^{-1}, \ M_{2} =n_{2}m_{2}'=M\prod_{j=0}^{1}A_{1, 6k-6j-3}^{-1},\
\ldots, \\
&M_{k-1} =n_{k-1}m_{2}'=M\prod_{j=0}^{k-2}A_{1, 6k-6j-3}^{-1}, \ M_{k} =n_{k}m_{2}'=M\prod_{j=0}^{k-1}A_{1, 6k-6j-3}^{-1},
\end{split}
\end{eqnarray*}
and every monomial above has multiplicity $1$ in $\chi_{q}(\mathcal T_{k, l+3}^{(0)})\chi_{q}(\mathcal T_{k, l}^{(6)})$.

{\bf Classify the dominant monomials in $\chi_q(T_{k+1, l}^{(s)}) \chi_q(T_{k-1, l+3}^{(s+6)})$.}

Let $m_{1}'=T_{k+1, l}^{(s)}$, $m_{2}'=T_{k-1, l+3}^{(s+6)}$. Without loss of generality, we may assume that $s=0$. Then
\begin{equation*}
\begin{split}
&m_{1}'=(1_{0} 1_{6} \cdots 1_{6k})(2_{6k+7} 2_{6k+9} \cdots 2_{6k+2l+5}), \\
&m_{2}'=(1_{6} 1_{12} \cdots 1_{6k-6})(2_{6k+1} 2_{6k+3} \cdots 2_{6k+2l+5}).
\end{split}
\end{equation*}
Let $m = m_{1}m_{2}$ be a dominant monomial, where $m_{i} \in \chi_{q}(m_{i}'), i=1, 2$. By the same argument as above, we have $m_1 = m_1'$ and $m_2$ is one of the following monomials.
\begin{align*}
& p_{1} = m_{2}'A_{1, 6k-3}^{-1} = 1_{0} 1_{6} \cdots 1_{6k-12} 1_{6k}^{-1} 2_{6k-5} 2_{6k-3} \cdots 2_{6k+2l+5}, \\
& p_{2} = m_{2}'A_{1, 6k-3}^{-1}A_{1, 6k-9}^{-1} = 1_{0} 1_{6} \cdots 1_{6k-18} 1_{6k-6}^{-1} 1_{6k}^{-1} 2_{6k-11} 2_{6k-9} \cdots 2_{6k+2l+5}, \\
& \cdots \\
& p_{k-1} = m_{2}'A_{1, 6k-3}^{-1}A_{1, 6k-9}^{-1} \cdots A_{1, 9}^{-1} = 1_{12}^{-1} \cdots 1_{6k-6}^{-1} 1_{6k}^{-1} 2_{7} 2_{9} \cdots 2_{6k+2l+5}.
\end{align*}
It follows that the dominant monomials in $\chi_q(T_{k+1, l}^{(0)}) \chi_q(T_{k-1, l+3}^{(6)})$ are
\begin{eqnarray*}
\begin{split}
&M= m_{1}'m_{2}', \ M_{1} =m_{1}'p_{1}=MA_{1, 6k-3}^{-1}, \ M_{2} =m_{1'}p_{2}=M\prod_{j=0}^{1}A_{1, 6k-6j-3}^{-1},\
\ldots, \\
&M_{k-1} =m_1'p_{k-1}=M\prod_{j=0}^{k-2}A_{1, 6k-6j-3}^{-1},
\end{split}
\end{eqnarray*}
and every dominant monomial has multiplicity one in $\chi_q(T_{k+1, l}^{(0)}) \chi_q(T_{k-1, l+3}^{(6)})$.

{\bf Classify the dominant monomials in $\chi_q(T_{0, 3k+l+3}^{(s)}) \chi_q(T_{0, l}^{(s+6k+6)})$.}

Let $m_{1}'=T_{0, 3k+l+3}^{(s)}$, $m_{2}'=T_{0, l}^{(s+6k+6)}$. Without loss of generality, we may assume that $s=0$. Then
\begin{equation*}
\begin{split}
&m_{1}'= 2_1 2_3 \cdots 2_{6k+2l+5}, \\
&m_{2}'= 2_{6k+7} 2_{6k+9} \cdots 2_{6k+2l+5}.
\end{split}
\end{equation*}
Let $m = m_{1}m_{2}$ be a dominant monomial, where $m_{i} \in \chi_{q}(m_{i}'), i=1, 2$. By Lemma \ref{non highest monomials in q characters of KR modules are right negative}, if $m_1 \neq m_1'$, then $m_1$ is right negative. The index of the negative factor in $m_1$ with largest index is greater than $6k+2l+5$. If $m_2 = m_2'$, then the negative factor with largest index in $m_1$ cannot be canceled by $m_2$. Therefore $m=m_1m_2$ is not dominant which contradicts our assumption. Hence $m_2 \neq m_2'$. Therefore by Lemma \ref{non highest monomials in q characters of KR modules are right negative}, $m_2'$ is right negative. It follows that $m=m_1 m_2$ is right negative since both of $m_1$ and $m_2$ are right negative. This is a contradiction. Therefore $m_1=m_1'$.

If $m_2 \neq m_2'$, then $m_2$ is right negative and $m=m_1 m_2$ is right negative. This is a contradiction. Therefore $m_2=m_2'$. It follows that the only dominant monomial in $\chi_q(T_{0, 3k+l+3}^{(0)}) \chi_q(T_{0, l}^{(6k+6)})$ is $T_{0, 3k+l+3}^{(0)} T_{0, l}^{(6k+6)}$ and $T_{0, 3k+l+3}^{(0)} T_{0, l}^{(6k+6)}$ has multiplicity one in $\chi_q(T_{0, 3k+l+3}^{(0)}) \chi_q(T_{0, l}^{(6k+6)})$.
\end{proof}

\subsection{Proof of the equations in Theorem $\ref{first part of the M system of type G2}$}
By Lemma $\ref{lemma1}$, the dominant monomials in the $q$-characters of the left hand side and of the right hand side of every equation in Theorem $\ref{first part of the M system of type G2}$ are the same and have the same multiplicities. Therefore by Proposition \ref{dominant monomials determine q-characters}, the theorem is true.

\section{Proof of the simplicity of the modules in the summands on the right hand side of the equations in Theorem \ref{first part of the M system of type G2}}  \label{proof simplicity of the modules on the right hand side}

By Lemma $\ref{lemma1}$, the modules corresponding to the second summand of every equation in Theorem $\ref{first part of the M system of type G2}$ are special and hence they are simple. We only need to show that the modules in the first summand corresponding to every equation in Theorem $\ref{first part of the M system of type G2}$ are simple. Let $\mathcal{S}$ be a module corresponding to the first summand corresponding to an equation in Theorem $\ref{first part of the M system of type G2}$. It suffices to prove that for each non-highest dominant monomial $M$ in $\mathcal{S}$, we have $\chi_q(L(M))\not \subseteq \chi_q(\mathcal{S})$, see \cite{H06}, \cite{MY12a}.

\begin{lemma}
We consider the same cases as in Lemma $\ref{lemma1}$. In each case $M_{i}$ are the dominant monomials described by that Lemma $\ref{lemma1}$.

\begin{enumerate}[(1)]
\item For $k\in \mathbb{Z}_{\geq 1}$, $l\in \{1, 2, 3\}$, let
\begin{equation*}
\begin{split}
&t_{i} = M_{i}A_{1,s+6k-6i+3}^{-1}, \quad i=1, 2, \ldots, k-1.
\end{split}
\end{equation*}
Then for $i=1, 2, \cdots, k-1$, $t_{i} \in \chi_{q}(M_{i})$ and $t_{i}\not \in \chi_{q}(\mathcal T_{k+1, 0}^{(s)})\chi_{q}(\mathcal T_{k-1, l}^{(s+6)})$.

\item For $k, l\in \mathbb{Z}_{\geq 1}$, let
\begin{equation*}
\begin{split}
&t_{i} = M_{i}A_{1,s+6k-6i+3}^{-1}, \quad i=1, 2, \ldots, k-1.
\end{split}
\end{equation*}
Then for $i=1, 2, \cdots, k-1$, $t_{i} \in \chi_{q}(M_{i})$ and $t_{i}\not \in \chi_{q}(\mathcal T_{k+1, l}^{(s)})\chi_{q}(\mathcal T_{k-1, l+3}^{(s+6)})$.
\end{enumerate}
\end{lemma}

\begin{proof}
We will prove part (2). Part (1) is similar. Without loss of generality, we may assume that $s=0$. By definition, we have
\begin{equation*}
\begin{split}
& T_{k+1, l}^{(0)} = 1_{0} 1_{6} \cdots 1_{6k-6} 1_{6k} 2_{6k+7} 2_{6k+9} \cdots 2_{6k+2l+5}, \\
& T_{k-1, l+3}^{(6)} = 1_{6} 1_{ 12} \cdots 1_{6k-6} 2_{6k+1} 2_{6k+3} \cdots 2_{6k+2l+5}.
\end{split}
\end{equation*}
Let $i \in \{1,2,\ldots,k-1\}$. Then
\begin{equation*}
\begin{split}
M_{i} & = M \prod_{j=0}^{i-1}A_{1, 6k-6j-3}^{-1} \\
& = T_{k+1, l}^{(0)} T_{k-1, l+3}^{(6)} \prod_{j=0}^{i-1}A_{1, 6k-6j-3}^{-1} \\
& =  1_{0} 1_{6}^2 \cdots 1_{6k-6i-6}^2 1_{6k-6i} 2_{6k-6i+1} 2_{6k-6i+3} \cdots 2_{6k+5} 2_{6k+7}^2 \cdots 2_{6k+2l+5}^2.
\end{split}
\end{equation*}

By Theorem \ref{Uqsl_2 arguments}, the monomial
\begin{align*}
& M_{i}A_{1,6k-6i+3}^{-1} \\
& \scalemath{0.9}{
= 1_{0} 1_{6}^2 \cdots 1_{6k-6i-6}^2 1_{6k-6i+6}^{-1} 2_{6k-6i+1}^2 2_{6k-6i+3}^2 2_{6k-6i+5}^2 2_{6k-6i+7} 2_{6k-6i+9} \cdots 2_{6k+5} 2_{6k+7}^2 \cdots 2_{6k+2l+5}^2
}
\end{align*}
is in $\chi_{q}(M_{i})$.

We have
\begin{align*}
t_{i} & = M_{i}A_{1,6k-6i+3}^{-1} \\
& = \left( T_{k+1, l}^{(0)} T_{k-1, l+3}^{(6)} \prod_{j=0}^{i-1}A_{1, 6k-6j-3}^{-1} \right) A_{1,6k-6i+3}^{-1} \\
& = \left( T_{k+1, l}^{(0)} A_{1,6k-6i+3}^{-1} \right) \left(  T_{k-1, l+3}^{(6)} \prod_{j=0}^{i-1}A_{1, 6k-6j-3}^{-1} \right).
\end{align*}

By Theorem \ref{Uqsl_2 arguments}, the monomial
\begin{align*}
& T_{k-1, l+3}^{(6)} \prod_{j=0}^{i-1}A_{1, 6k-6j-3}^{-1} \\
& = 1_{6} 1_{ 12} \cdots 1_{6k-6} 2_{6k+1} 2_{6k+3} \cdots 2_{6k+2l+5} \prod_{j=0}^{i-1}A_{1, 6k-6j-3}^{-1} \\
& = 1_{6} 1_{ 12} \cdots 1_{6k-6i-12} 1_{6k-6i-6} 1_{6k-6i+6}^{-1} 1_{6k-6i}^{-1} \cdots 1_{6k}^{-1} 2_{6k-6i+1} 2_{6k-6i+3} \cdots 2_{6k+2l+5}
\end{align*}
is in $\chi_q(T_{k-1, l+3}^{(6)})$. Since $1_{6k-6i}$ is not a factor of $T_{k-1, l+3}^{(6)} \prod_{j=0}^{i-1}A_{1, 6k-6j-3}^{-1}$ (this monomial is in $\chi_q(T_{k-1, l+3}^{(6)})$), we have that the monomial $\left(T_{k-1, l+3}^{(6)} \prod_{j=0}^{i-1}A_{1, 6k-6j-3}^{-1}\right) A_{1,6k-6i+3}^{-1}$ is not in $\chi_q(T_{k-1, l+3}^{(6)})$ by the Frenkel-Mukhin algorithm.

Therefore if
\begin{align*}
t_{i}=\left( T_{k+1, l}^{(0)} A_{1,6k-6i+3}^{-1} \right) \left(  T_{k-1, l+3}^{(6)} \prod_{j=0}^{i-1}A_{1, 6k-6j-3}^{-1} \right)
\end{align*}
were in $\chi_{q}(\mathcal T_{k+1, l}^{(0)})\chi_{q}(\mathcal T_{k-1, l+3}^{(6)})$, then $T_{k+1, l}^{(0)} A_{1,6k-6i+3}^{-1}$ would be in $\chi_{q}(T_{k+1, l}^{(0)})$. This implies that $T_{k+1, l}^{(0)} A_{1,6k-6i+3}^{-1} \in \varphi_1(T_{k+1, l}^{(0)})$, where the map $\varphi_1$ is defined before Theorem \ref{Uqsl_2 arguments}, which contradicts Lemma \ref{need to expand from right most factors}: $\beta_1(T_{k+1,l}^{(0)})=1_{0} 1_{6} \cdots 1_{6k}$ is a $q_1$-string in $T_{k+1, l}^{(0)}$, $1_{6k}$ is a factor of $T_{k+1, l}^{(0)} A_{1,6k-6i+3}^{-1}$, but $\beta_1( T_{k+1, l}^{(0)} A_{1,6k-6i+3}^{-1} ) \neq \beta_1(T_{k+1, l}^{(0)})$. Therefore $t_i$ is not in $\chi_{q}(\mathcal T_{k+1, l}^{(0)})\chi_{q}(\mathcal T_{k-1, l+3}^{(6)})$.
\end{proof}

\section{Proof of Theorem \ref{second part of M system}} \label{sec: second part of M system}
In this section, we prove Theorem \ref{second part of M system}.

\begin{theorem}[{Theorem 7.2, \cite{LM13}}]
The module $\widetilde{\mathcal T}_{k, l}^{(s)}$, $s \in \mathbb{Z}, k, l\in \mathbb{Z}_{\geq 0}$ are anti-special.
\end{theorem}

\begin{lemma}[{Lemma 7.3, \cite{LM13}}]\label{lemma2}
Let $\iota: \mathbb{Z}\mathcal{P}\rightarrow \mathbb{Z}\mathcal{P}$ be a homomorphism of rings such that $Y_{1, aq^{s}}\mapsto Y_{1, aq^{12-s}}^{-1}$, $Y_{2, aq^{s}}\mapsto Y_{2, aq^{12-s}}^{-1}$ for all $a \in \mathbb{C}^{\times}, s \in \mathbb{Z}$. Then
\begin{align*}
\chi_{q}(\widetilde{\mathcal T}_{k, l}^{(s)})=\iota(\chi_{q}(\mathcal T_{k, l}^{(s)})).
\end{align*}
\end{lemma}

\begin{proof}[{\bf Proof of Theorem \ref{second part of M system}}]
The lowest weight monomial of $ \chi_q(\mathcal{T}_{k,l}^{(s)}) $ is obtained from the highest weight monomial of $ \chi_q(\mathcal{T}_{k,l}^{(s)}) $ by the substitutions: $1_s \mapsto 1^{-1}_{12+s}$, $2_s \mapsto 2^{-1}_{12+s}$. After we apply $\iota$ to $\chi_{q}(\mathcal T_{k, l}^{(s)})$, the lowest weight monomial of $ \chi_q(\mathcal{T}_{k,l}^{(s)}) $ becomes the highest weight monomial of $\iota(\chi_{q}(\mathcal T_{k, l}^{(s)}))$. Therefore the highest weight monomial of $\iota(\chi_{q}(\mathcal T_{k, l}^{(s)}))$ is obtained from the lowest weight monomial of $ \chi_q(\mathcal{T}_{k,l}^{(s)}) $ by the substitutions: $1_s \mapsto 1^{-1}_{12-s}$, $2_s \mapsto 2^{-1}_{12-s}$. It follows that the highest weight monomial of $\iota(\chi_{q}(\mathcal T_{k, l}^{(s)}))$ is obtained from the highest weight monomial of $ \chi_q(\mathcal{T}_{k,l}^{(s)}) $ by the substitutions: $1_s \mapsto 1_{-s}$, $2_s \mapsto 2_{-s}$. Therefore the second part of the M-system is obtained from the first part of the M-system by applying $\iota$ to both sides of every equation in the first part of the M-system.

The simplicity of every module corresponding to the summands on the right hand side of every equation in Theorem \ref{second part of M system} follows from the simplicity of the modules corresponding to the summands on the right hand side of the equations in Theorem \ref{first part of the M system of type G2} and Lemma $\ref{lemma2}$.
\end{proof}

\section*{Acknowledgements}
The authors are very grateful to the anonymous referee for the comments and suggestions that have been very helpful to improve the quality of this paper. J.-R. Li is very grateful to Professors David Hernandez, Bernard Leclerc, and Evgeny Mukhin for stimulating discussions. The research of J.-R. Li on this project is supported by the Minerva foundation with funding from the Federal German Ministry for Education and Research; the National Natural Science Foundation of China (no. 11501267, 11371177, 11401275); ERC AdG Grant 247049; the PBC Fellowship Program of Israel for Outstanding Post-Doctoral Researchers from China and India; the European Research Council (ERC) under the European Union's Horizon 2020 research and innovation programme (QUASIFT grant agreement 677368). J.-R. Li is very grateful to Institut des Hautes Etudes Scientifiques (IHES) and Mainz Institute for Theoretical Physics (MITP) for hospitality where a part of this work has been done.

\end{document}